\documentclass[reqno]{amsart}
\usepackage{epsfig,latexsym,amsfonts,amssymb,amsmath,amscd}
\usepackage{amsthm}
\usepackage{tikz-cd}
\usepackage[mathscr]{eucal}
\usepackage{mathrsfs}
\usepackage{mathbbol}
\usepackage{array}
\usetikzlibrary{matrix,arrows,decorations.pathmorphing}
\usepackage{oldgerm,units,color}

\usepackage{pst-node}
\usepackage{graphicx}

\usepackage{eepic, epic}

\usepackage{ifthen}
\DeclareMathOperator{\coker}{coker}

\def\tTp{ \tT^+}
\def\tTm{ \tT^-}
\def\tTz{ \tT_\operatorname{\zero}}

\newcommand{\too}[1]{\overset{#1}\longrightarrow}

\def\module0{module$^\dagger$}
\def\Null{{\operatorname{Null}}}

\def\ker{\operatorname{ker}}
\def\pker{\operatorname{preker}}

\def\NTkerC{ \operatorname \ker_{\operatorname{N}}}

\def\ssemiring0{$s$-semiring$^\dagger$}
\newtheorem*{nothma}{\textbf{Theorem A}}
\newtheorem*{nothmb}{\textbf{Theorem B}}
\newtheorem*{nothmc}{\textbf{Theorem C}}
\newtheorem*{nothmd}{\textbf{Theorem D}}

\newtheorem*{nothmf}{\textbf{Theorem F}}
\newtheorem*{nothmg}{\textbf{Theorem G}}

\newtheorem*{oprblm}{\textbf{Open Problem}}

\usepackage{tikz}
\usepackage{epsfig,latexsym,amsfonts,amssymb,amsmath,amscd,graphics,epic}
\usepackage{amsfonts,amssymb,amsmath,amscd,amsthm}
\usepackage[mathscr]{eucal}
\usepackage{mathrsfs}

\usepackage{mathbbol}

\usepackage{oldgerm,units}
\usepackage{wrapfig,epsfig}

\usepackage{ifthen}
\usepackage{mathbbol}

\usepackage{amsthm}
\usepackage[mathscr]{eucal}
\usepackage{mathrsfs}
\usepackage{mathbbol}
\usepackage{oldgerm,units}
\usepackage{wrapfig}

\newtheorem{theorem}{Theorem}[section]

\newtheorem{definition}[theorem]{Definition}

\newtheorem{lemma}[theorem]{Lemma}

\newtheorem{example}[theorem]{Example}
\newtheorem{remark}[theorem]{Remark}

\newcommand{\one}{\mathbb{1}}
\newcommand{\zero}{\mathbb{0}}

\newcommand{\trop}[1]{\mathcal{#1}}

\newcommand{\tG}{\trop{G}}

\newcommand{\tT}{\trop{T}}

\newcommand{\Hom}{Hom}

    \ifx\proof\undefined
    \newenvironment{proof}{
    \smallskip
    \noindent\emph{Proof.}}{\hfill\(\Box\)
    \bigskip
    } \fi

\newcommand{\ifdef}[3]{\ifthenelse{\equal{#1}{true}}{#2}{#3}}

\definecolor{lgray}{gray}{0.90}

\def\Mod{\textbf{Mod}}

\def\BMod{$\mathbb B$-\textbf{Mod}}

\def\ctw{\cdot_{\operatorname{tw}}}

\def\vep{\varepsilon}

\def\mcS{\mathcal S}

\def\({\left(}
\def\){\right)}

\def\Z{{\mathbb Z}}
\def\Q{{\mathbb Q}}

\def\pipe{{\underset{{\ \, }}{\mid}}}

\def\vsemifield0{$\nu$-semifield$^\dagger$}
\def\vsemiring0{$\nu$-semiring$^\dagger$}

\def\pipe1{{\underset{{1}}{\mid}}}
\def\lmod1{\mathrel  \pipe1  \joinrel \joinrel =}

\def\CFunFF1{\operatorname{CFun} (F,F)}

\def\semiring0{semiring$^{\dagger}$}
\def\Semiring0{Semiring$^{\dagger}$}
\def\Semirings0{Semirings$^{\dagger}$}
\def\semidomain0{semidomain$^{\dagger}$}
\def\semifield0{semifield$^{\dagger}$}
\def\semifields0{semifields$^{\dagger}$}
\def\vsemifields0{$\nu$-semifields$^{\dagger}$}
\def\domain0{domain$^{\dagger}$}
\def\predomain0{pre-domain$^{\dagger}$}
\def\predomains0{pre-domains$^{\dagger}$}
\def\domains0{domains$^{\dagger}$}
\def\preceqh{\preceq_{\operatorname{hyp}}}
\def\vdomains0{$\nu$-domains$^{\dagger}$}

\def\domains0{domains$^\dagger$}

\newcommand{\etype}[1]{\renewcommand{\labelenumi}{(#1{enumi})}}
\def\eroman{\etype{\roman}}

\def\pipe{{\underset{{\tG}}{\mid}}}

\def\lmod{\mathrel  \pipe \joinrel \joinrel =}
\def\pipe{{\underset{{\tG}}{\mid}}}

\def\Ann{{\operatorname{Ann}}\,}

\newtheorem{thm}[theorem]{Theorem}

\newtheorem*{thm*}{Theorem}

\def\Hom{\operatorname{Hom}}

\newtheorem{lem}[theorem]{Lemma}
\newtheorem{rem}[theorem]{Remark}
\newtheorem{prop*}{Proposition}
\newtheorem{conj*}{Conjecture}

\newtheorem{prop}[theorem]{Proposition}
\newtheorem{defn}[theorem]{Definition}
\newtheorem*{examp*}{Example}
\newtheorem*{examples*}{Examples}
\newtheorem*{remark*}{Remark}
\newtheorem*{defn*}{Definition}
\newtheorem*{note*}{Note}

\def\la{\lambda}

\def\tT{\mathcal T}

\def\tTz{\tT_\zero}

\numberwithin{equation}{section}

\def\M0{M_{\zero}}

\def\AA{A}
\def\BA{B}

\def\PS{P}

\def\Cong{\Phi}

\def\Diag{{\operatorname{Diag}}}

\def\semirings0{semirings$^\dagger$}

\newcommand{\nPS}[1]{\PS_{(!#1)}}
\newcommand{\nPSo}[1]{\nPS{\one}}

\begin{document}


\title[Homology of systemic modules]
{Homology of systemic modules}


\author[J.~Jun]{Jaiung~Jun}
\address{Department of Mathematics, State University of New York at New Paltz, New Paltz, NY 12561, USA} \email{jujun0915@gmail.com}

\author[K.~Mincheva]{Kalina~Mincheva}
\address{Department of Mathematics
, Yale University, New Haven, CT 06511, USA}
\email{kalina.mincheva@yale.edu}

\author[L.~Rowen]{Louis Rowen}
\address{Department of Mathematics, Bar-Ilan University, Ramat-Gan 52900,
Israel} \email{rowen@math.biu.ac.il}

\makeatletter
\@namedef{subjclassname@2020}{%
	\textup{2020} Mathematics Subject Classification}
\makeatother

\subjclass[2020]{Primary  08A05,  14T10, 16Y60,  18A05, 18C10;
Secondary 08A30, 08A72, 12K10, 13C60, 18E05, 20N20}




\keywords{Category, system, triple, negation map, morphism, module,
symmetrization, congruence,  $(-)$-bipotent, height, idempotent,
supertropical algebra, tropical algebra,
  hypergroup,  exact
sequence, chain,   homology, projective, dimension, polynomial,
semiring, semifield, surpassing relation, semiexact category,
homological category.}

\thanks{\textbf{Acknowledgment:} J.J. was supported by an AMS-Simons travel
grant.}


\begin{abstract}
We develop the rudiments of a tropical homology theory, based on
``triples'' and ``systems.''  Results include a version of
Schanuel's lemma, projective dimension, the homology semi-module,
and a weak Snake lemma.
\end{abstract}

\maketitle

\tableofcontents




\section{Introduction}

 In this paper we explore homology theory for tropical mathematics.
This has been initiated in a strong paper of  Connes and Consani
\cite{CC2}, which we utilize, but we work more algebraically in the
 theory of triples and systems which was espoused in
\cite{JuR1, Row16}, as surveyed in \cite{Row17}, (and applied in
\cite{AGR, GaR,JMR}) in order to unify classical algebra with
 the algebraic theories of supertropical algebra, symmetrized semirings, hyperfields, and fuzzy rings.

 We briefly recall that we start with a set
  $\tT$ (often a multiplicative monoid), whose intrinsic additive structure is not
  rich enough to utilize established algebraic techniques. Two such
  examples related to tropical mathematics are
  the max-plus algebra, where we get into difficulties unless we eliminate the idempotent rule $a+a = a,$ and the hyperfield,
  whose addition requires passing to the power set; see \cite[\S 1.2]{Row17} for more
  examples.
   We embed  $\tT$ into
  a fuller algebraic structure $\mathcal A$ on which $\tT$ acts, which possesses a formal negation map
  $(-)$. ($\mathcal A, \tT, (-)$) is called a \textbf{triple}.   (\cite{CC2} utilizes a categorical ``involution'' which is much like the negation map
  here.)

The negation map is essential for developing a viable homology theory.
Semirings (such as idempotent semirings) often do not have a classical negation. 
In this paper we look at two natural ways to introduce a negation
map in the context of tropical mathematics. The first is simply
taking the negation map to be the identity, as done in supertropical
math \cite{IR} and idempotent mathematics.  The second is to
construct a negation map by a symmetrization functor, which recovers
much of \cite{CC2}, and moreover provides a decomposition into
positive and negative parts.

We consider a more general setting (allowing us to treat rings and
semirings at the same time), by replacing equality with a
``surpassing relation'' $\preceq$, given below in
Definition~\ref{precedeq07}. This relation restricts to equality on
$\tT$. It is reflexive and transitive but not symmetric on~$\mathcal
A$. A triple with such relation we call a \textbf{system}. The
surpassing relation $\preceq$ plays a key structural role.
\footnote{There is a related weaker relation $\nabla$, where $a_1
\nabla a_2$ if and only if $\zero \preceq a_1 (-) a_2$, which also
restricts to equality on $\tT$. The relation $\nabla$  is
reflexive and symmetric, but not transitive. 
According to experience, $\nabla$ has been more effective in linear
algebra \cite{AGG1,AGR}, but $\preceq$ serves better in developing
representation theory, and  features in \cite{JMR} as well as this
paper.}

In~\cite{JuR1} we have considered both ``ground systems'' which
often are semirings, and systemic modules over ground systems. In
this paper we investigate the rudiments of homology theory of
systemic modules over a given ground system. Although there is
justification in limiting the class of modules, as in
\cite{IzhakianKnebuschRowenSA}, here we consider all the modules.
\cite{CC2} already has a categorical theory, with the emphasis on
homological categories. Although one can just consider their
homomorphisms, the major role of the surpassing map $\preceq$
naturally leads us to study ``$\preceq$-morphisms,'' in which
$f(a+a')\preceq f(a) +f(a')$. These have already shown  up in the
hyperfield literature, and have a natural structure theory which
parallels the classical theory of homomorphisms. 
\par\medskip

In this paper we study chains of modules and long exact sequences, in terms of  the surpassing relation $\preceq$, and provide a systemic version of \cite{AIKN}.

One major obstacle in obtaining a meaningful homology theory for semirings is the lack of correspondence between ideals (modules) and kernels of homomorphisms (congruences). Classically, homology is defined to be "the kernel modulo the image" in a chain. In the semiring case however, the kernel is a congruence and the image does not necessarily correspond to the kernel of a homomorphism (i.e. is not a normal object), which makes the factor object hard to define at best. For this reason we need to work with congruences which are considerably more complicated, and cannot be viewed as morphisms in the original module category. Moreover, we often have to generalize congruences to precongruences, in order to overcome some of the difficulties in handling images and cokernels.

The discrepancy between the category of modules already has been addressed in the literature \cite{CC2,Flo,grandis2013homological}. One approach to homology is the use of the double arrow chain complex introduced by Patchkoria \cite{Pat} or Flores \cite{Flo}. However, Connes and Consani \cite{CC2} provide some troublesome examples \cite[Examples~4.9,4.14]{CC2} concerning this approach.

A categorical approach is taken by Connes and Consani \cite[Theorem~6.12]{CC2} who obtain as decisive a theorem as possible for modules over the Boolean semifield. In \S \ref{Grand} we obtain part of \cite[Theorem~6.12]{CC2} in the systemic setting of $\circ$-idempotent modules.

In this paper, we search for the precise algebraic notions which will permit us to obtain explicit versions of standard homological results such as Schanuel's lemma and the Snake Lemma. It does not seem possible to get the full analog of these results, but there are convincing partial results, which provide insight not available in the categorical approach.


\subsection{Main Results}$ $

  In \cite{JMR} we studied versions of projective modules and Schanuel's Lemma.
Here we continue with projective dimension via a version of
Schanuel's Lemma which is more technical but which fits in well with
the systemic environment, i.e., by means of negation maps and
symmetrization as in \cite{AGR}, \cite{Row16}, \S\ref{symmod}. Then
we study homology semigroups of chains of modules, in terms of
congruences and  $\preceq$-precongruences,  avoiding the pitfalls of
Remark~\ref{congprob}. We conclude by bringing in the categorical
viewpoint in  \S\ref{Grand}.\par\medskip

In order to state the main results we need the following definitions. Let the map $f:\mathcal M \to \mathcal M'$ be a $\preceq$-morphism of systemic modules over a semiring system~$\mathcal A$.
\begin{enumerate}
    \item The \textbf{null-module kernel} denoted $\ker_{\Mod,\mathcal M} f$ is defined as the preimage of the set $\{a \in \mathcal{M'}: a \succeq \zero\}.$
    \item The $\succeq$-\textbf{module image} of $f$ is the set $f(\mathcal M )_\succeq :=  \{ b'\in  \mathcal M' : f(b)      \succeq  b',\text{ for some } b\in  \mathcal M \}.$
    \item The \textbf{systemic cokernel} of $f: \mathcal M \to \mathcal M'$, denoted $\coker f_{\operatorname{sys}}$, is the set $\{[b]_f : b \in \mathcal M'\},$ where $[b]_f = \{ b' \in \mathcal M' : f(c) \preceq b' (-)b$, for some $c \in \mathcal M\}.$

\end{enumerate}

\begin{nothma} [Theorems~\ref{morphsym}, \ref{morphsym7}]
For any $\tT$-module $\mathcal A$, we can embed  $\mathcal
 A$ into $\widehat {\mathcal A}:= \mathcal A \oplus \mathcal A$ via $b \mapsto (b,\zero)$, thereby
 obtaining a faithful functor from the category of semirings into  the category of
 semirings with a negation map (and preserving additive
 idempotence). This yields a faithful functor from ordered semigroups to  signed $(-)$-bipotent
 systems. Any $\mathcal A$-module $\mathcal M$ yields a $\widehat{\mathcal A}$-module $\widehat{\mathcal M}= \mathcal M \oplus \mathcal M,$
 which has a signed
 decomposition   where $\mathcal M^+$ is the
 first component.
\end{nothma}

\begin{nothmb} [Theorem~\ref{Sch}] (Semi-Schanuel, homomorphic-version)\label{Sch1}
 Let $\mathcal P   \overset {f }
\longrightarrow \mathcal M$ and $\mathcal P'  \overset {f'}
\longrightarrow \mathcal M'$ be $\preceq$-onto homomorphisms, where
$\mathcal P$ is $h$-projective  and $\mathcal P'$ is
$(\preceq,h)$-projective, and let $\mathcal K' = \ker_{\Mod,\mathcal
P'}f'.$ There is a $\preceq$-onto $\preceq$-splitting homomorphism $g:
\mathcal K' \oplus \mathcal P \to \mathcal P'$, with a
$\preceq$-isomorphism (i.e., $\preceq$-monic and $\preceq$-onto)
$\Phi: \mathcal K \to \mathcal K'' $, where $\mathcal K'' = \{
(b,b') \in \ker_{\Mod,\mathcal K' \oplus \mathcal P}g: b \preceq
\tilde \mu(b')\}$.
\end{nothmb}

%


Next, we prove that a ``weak version'' of the Snake Lemma holds in our setting. Consider the following commutative diagram of $\preceq$-morphisms:
\begin{equation}\label{eq: sqr}
\begin{tikzcd}
\mathcal M' \arrow{r}{q}\arrow{d}{f} & \mathcal N'\arrow{d}{g}\arrow{r}{p} & \mathcal{L}'\arrow{d}{h}
\\
\mathcal M\arrow{r}{l} & \mathcal N \arrow{r}{r} & \mathcal{L}
\end{tikzcd}
\end{equation}

\begin{nothmc} [Theorem~\ref{snake1}]
Under technical conditions for \eqref{eq: sqr}, there exists a
natural $\tT_\mathcal{A}$-equivariant map $d: \ker_{\Mod,\mathcal
L'} h \to \coker f_{\operatorname{sys}}$ (given in the proof). An
analogous systemic version  provides a natural $\preceq$-morphism
$d: \ker_{\Mod,\mathcal L'} h \to \mathcal M$, if the $f(\mathcal
M')$-minimality condition of Definition \ref{reve1} holds.
\end{nothmc}

\begin{nothmd} (Weak Snake Lemma) [Lemma~\ref{chai}, Theorem~\ref{theorem: snake lemma}]
Under technical conditions for \eqref{eq: sqr}, we have the sequence
of $\tT_\mathcal{A}$-modules with $\tT_{\mathcal{A}}$-equivariant
maps
\[
    \ker_{\Mod,\mathcal M'}{f} \too {\tilde q} \ker_{\Mod,\mathcal
        N'}{g} \too {\tilde p} \ker_{\Mod,\mathcal L'} h \overset{d}
    \longrightarrow \coker({f})_{\operatorname{sys}} \too {\bar
        l}\coker({g})_{\operatorname{sys}}\too {\bar
        r}\coker({h})_{\operatorname{sys}},
\]
satisfying the following properties:
\begin{enumerate} \eroman
    \item
    If the top row of \eqref{eq: sqr} is exact, then
    \[
    \tilde{q}(\ker_{\Mod,\mathcal M'}{f})_\succeq = \ker_{\Mod,\mathcal H}{\tilde{p}},
    \]
    where $\mathcal {H}=\ker_{\Mod,\mathcal
        N'}{g}$.
    \item
    \[
    \tilde{p}(\ker_{\Mod,\mathcal N'}{g}) \subseteq \{b \in \ker_{\Mod,\mathcal
        L'} h : d(b)=[\zero]\},
    \]
    where $[\zero]$ is the equivalence class of $\zero$ in $\coker(f)_{\operatorname{sys}}$.

    \item
    \[
    d(\ker_{\Mod,\mathcal L'} h) \subseteq \{[b] \in \coker(f)_{\operatorname{sys}} :
    \bar{l}([b])=[\zero]\},
    \] where $[\zero]$ is the equivalence class
    of $\zero$ in $\coker(g)$.

    \item
    If the bottom row of \eqref{eq: sqr} is exact, then
    \[
    \bar{l}(\coker(f)_{\operatorname{sys}}) \subseteq \{b' \in
    \coker(g)_{\operatorname{sys}} : \bar{r}(b')=[\zero]\},
    \]
    where $[\zero]$ is the equivalence class of $\zero$ in $\coker(h)_{\operatorname{sys}}$.
\end{enumerate}
\end{nothmd}


\begin{nothmf}[Theorem~\ref{catview}]
Suppose $\mathcal{C}$ is a pointed category. Define $\widehat {\mathcal{C}}$ to have
\begin{enumerate}
    \item
    Objects; the pairs $(A,A)$ for each object $A$ of $\mathcal C$.
    \item
    Morphisms; for $(A,A), (B,B) \in \emph{Obj}(\widehat{\mathcal C})$,
    \[
    \Hom_{\widehat{\mathcal{C}}}((A,A),(B,B))=\{(f,g) \mid f,g \in \Hom_\mathcal{C} (A,B)\}.
    \]
\end{enumerate}
The composition of two morphisms are given by the \textbf{twist
    product}, that is, for $(f_1,f_2):(B,B) \to (C,C)$
and $(g_1,g_2):(A,A) \to (B,B)$,
\[
(f_1,
f_2)(g_1,g_2) = ( f_1 g_1 + f_2 g_2,  f_2 g_1+  f_1 g_2)
\]
Define the \textbf{switch} $(-)_{\operatorname{sw}}(f_1, f_2) = (f_2,f_1)$. Then $\widehat {\mathcal{C}}$ is a category with negation functor $(-)_{\operatorname{sw}},$ and there is a faithful functor $F:{\mathcal{C}} \to \widehat {\mathcal{C}}$ such that $F(A)=(A,A)$ for an object $A$, and $F(f)=(f,0_A)$ for $f \in \Hom_\mathcal{C} (A,B)$.
\end{nothmf}

\begin{nothmg}[Theorem~\ref{proposition: modules semiexact}]
Let $(\mathcal A,\tT, (-))$ be a triple, and
$\textbf{Mod}_{\mathcal{A},h}$ be the subcategory of
$\textbf{Mod}_\mathcal{A}$ with:
    \begin{enumerate}
        \item
    Objects: The 
    $\mathcal{A}$-modules $\mathcal M$ for which $(b^\circ)^\circ = b^\circ$ for each $b \in \mathcal M$.

       \item
    Morphisms: homomorphisms of $\mathcal{A}$-modules.
    \end{enumerate}
    With $N$ as in Example \ref{example: A-mod}, $\textbf{Mod}_{\mathcal{A},h}$ is a semiexact category.
\end{nothmg}

\par\medskip

\section{Basic notions}\label{BN}

  See \cite{Row17} for a relatively brief introduction; more
details are given in \cite{JMR}, \cite{JuR1}, and \cite{Row16}.
Throughout the paper, we let $\mathbb{N}$ be the additive monoid of
nonnegative integers. Similarly, we view $\mathbb{Q}$
(resp.~$\mathbb{R}$) as the additive monoid of rational numbers
(resp.~real numbers).

A \textbf{semiring} (cf.~\cite{Cos},\cite{golan92}) $(\mathcal A, +,
\cdot, 1)$ is an additive commutative semigroup $(\mathcal A, +,
\zero)$ and multiplicative monoid $(\mathcal A, \cdot, \one)$
satisfying $\zero b = b \zero = \zero$ for all $b \in \mathcal A$,
as well as the usual distributive laws.

The semiring predominantly used in tropical mathematics has been the
max-plus algebra, where  $\oplus$ designates $\max$, and $\odot $
designates $+$. However,
 we proceed with the familiar algebraic notation of addition and
multiplication in whichever setting under consideration. A
\textbf{magma} $\tT$ is a set with binary multiplication
\footnote{In order not to get caught up in later complications, but
not  following the convention in \cite[Definition 1.1]{Bo}, we
 adjoin a formal absorbing element $\zero$
 to $\tT$, i.e.,  $a\zero =  \zero a =  \zero $, for all $a \in
\tT$.}. We need not assume that multiplication in $\mathcal A$ is
 associative, so that we may treat ``Lie algebra homology'' in later
 work.
\subsection{$\tT$-modules}$ $

 We review the definitions throughout the next three subsections for the reader's
convenience. We assume that $\tT$ is a magma.

\begin{defn}\label{modu2}
A (left) $\tT$-\textbf{module}
 is a set $( \mathcal S,+,\zero)$ with a zero element (denoted by $\zero$) and scalar
multiplication $\tT\times \mathcal S \to \mathcal S$ satisfying the
following axioms, for all $a_i\in \tT$ and $b, b_j \in \mathcal S$:

\begin{enumerate}\eroman
\item
$(a_1 a_2) b = a_1(a_2 b).$
\item
$a \zero=\zero a = \zero.$
\item
$a (\sum _{j=1}^u b_j) = \sum _{j=1}^u (a b_j), \ a \in \mathcal A,$
\item
Let $\mathcal S_1$ and $\mathcal S_2$ be $\tT$-modules. A homomorphism $f: \mathcal S_1 \to \mathcal S_2$ is a function such that $f(ta)=tf(a)$, $f(a+b)=f(a)+f(b)$ for all $t \in \tT$ and $a,b \in \mathcal S_1$.
\end{enumerate}


\end{defn}

%


\begin{rem}
One example treated in the literature is the property that $b+b' =
\zero$ implies $b = b' = \zero,$ called \textbf{zero sum free}
in~\cite{golan92} and \textbf{lacking zero sums}  in \cite{IKR6},
also treated in \cite{CC2}. This property holds in tropical
mathematics, as well as numerous other situations, as indicated  in
\cite[Examples~1.9]{IKR6}. Using this property one can recover some
analogues to classical results such as all strongly projective modules
being direct sums of cyclic strongly projective modules; However,
this assumption is too strong to result in a useful homology theory, and
will not be pursued in this paper.
\end{rem}

\subsection{Triples}$ $

\begin{defn}\label{negmap}
 A \textbf{negation map} on  a $\tT$-module $\mathcal A$
is a semigroup isomorphism $(-)_\mathcal A :\mathcal A \to \mathcal
A$ of order~$\le 2,$  written $a\mapsto (-) a$, together with a map
$(-)_\mathcal T$  of order~$\le 2$ on $ \tT$ which also
 respects the $\tT$-action in the sense that
$$((-)_\mathcal T a)b = (-) (ab) = a ((-) b)$$ for $a \in \tT,$ $b \in \mathcal A.$ When the context is clear, we drop the
subscripts and simply denote the negation map by~$(-)$.
\end{defn}

\begin{lem}\label{circget11}
Let $\mathcal A$ be a $\tT$-module with a negation map. Then
$(-)\zero = \zero.$
\end{lem}
\begin{proof} Pick any $a\in \tT.$ Then $(-)\zero = (-)(a\zero) = ((-)(a))\zero
=
\zero.$
\end{proof}
We write $ b_1 (-)b_2 $ for $b_1+ ((-)b_2)$, and $ b ^\circ$ for $ b
(-)b$, called a \textbf{quasi-zero}. $(-b)$ is called the
 \textbf{quasi-negative} of $b$.
 An assortment of negation maps is given in
\cite{GaR,JuR1,Row16}.
 When $\one \in \tT$, the
negation map is given simply
 by $(-)b = ((-)\one)_{\tT} b$
for $b \in \mathcal A.$

\begin{rem}
Let $\mathcal A$ be a $\tT$-module with a negation map. Then
    $$ (-) b ^\circ=(-)(b(-)b)=   ((-)b)+b  =   b ^\circ.$$
In other words, quasi-zeros are fixed under any negation map.
\end{rem}

 The set $\mathcal A
^\circ$ of quasi-zeros is a $\tT$-submodule of $\mathcal A $ that
plays an important role. When $\mathcal A$ is a semiring,  $\mathcal
 A^\circ$ is an ideal.

\begin{note*}
In this paper, we always assume that   $\mathcal A$ is a
$\tT$-module, and also
 assume that $\tT$ is a multiplicative monoid,
leaving the Lie theory for later.
 When $\mathcal A$ is a
semiring, we essentially have Lorscheid's blueprints,
\cite{Lor1,Lor2}. We can make a $\tT$-module with associative
multiplication into a semiring by means of
\cite[Theorem~2.9]{Row16}.

In general, $\tT$ need not be a subset of $\mathcal A$. For example,
let $\tT=\mathbb{N}$ and $\mathcal{A}$ be any monoid. In this case,
$\mathcal A$ is equipped with a natural $\tT$-module structure,
although $\tT$ is not a subset of $\mathcal A$. In what follows, we
will further assume that $\tT$ is indeed a subset of $\mathcal A$,
starting from the following definition.\end{note*}

 \begin{defn}[\cite{Row16}]\label{sursys0}
A  \textbf{pseudo-triple} $(\mathcal A, \tT, (-))$ is a $\tT$-module
$\mathcal A$, with $\tT \subseteq \mathcal A$, called the set of
\textbf{tangible elements}, and a negation map $(-)$ on $\mathcal
A$. We write $\tTz$ for $\tT \cup \{\zero\}.$

 A
\textbf{triple} is a $\tT$-pseudo-triple, in which $\tT \cap
\mathcal A^ \circ = \emptyset$ and $\tTz$ generates $( \mathcal
A,+).$ 
\end{defn}

\begin{defn}\label{definition: height 2 module} \begin{enumerate}
 \eroman
 \item
Recall from \cite[Definition~1.28]{Row16}   that $\tT$-triple
$(\mathcal A, \tT, (-))$  is \textbf{metatangible} if $a + b \in
\tT$ whenever $a, b \in \tT$ with $b \neq (-) a.$ The main special
case: we call a triple $(-)$-\textbf{bipotent} if $a+ a' \in
\{a,a'\}$ whenever $a' \ne (-) a.$

 \item A  $(-)$-bipotent triple $(\mathcal A, \tT, (-))$ has
\textbf{height} 2 if $\mathcal A =   \tT \cup  \mathcal A^ \circ$.
  \end{enumerate}
   \end{defn}

Both the supertropical semiring \cite{IR} and the symmetrization of
an ordered group (viewed as an idempotent semiring), cf.
\S\ref{symmod} below, are  $(-)$-bipotent of height 2.

\begin{lem} For any  $(-)$-bipotent triple $(\mathcal A, \tT, (-))$
of \textbf{height} 2, $b^\circ +b=b ^\circ $ for all $b \in \mathcal
A$, and thus also in any free $\mathcal A$-module.
\end{lem}
\begin{proof} If $b \ne (-)b$ then $b^\circ +b = b(-)b+b = b (-)b
= b^\circ .$ The last assertion is by checking components.
 \end{proof}

\subsubsection{Triples of the  first and second kind}$ $

The triple is of the \textbf{first kind} if $(-)$ is the identity,
  of the \textbf{second kind} if $(-)$ is not the identity.

  When a
given $\tT$-module $\mathcal A$ does not come equipped with a
negation map, there are two main ways of providing one: Either take
$(-)$ to be the identity (first kind), as is done in supertropical
algebra  (and implicitly in much of the tropical literature), or we
``symmetrize''~$\mathcal A$ as in \S\ref{symmod}
below (second kind), as elaborated in \cite{AGR} and \cite{Ga}. 

\subsubsection{Signed decompositions}$ $

We will need a further refinement for triples of the   second kind,
taken from~\cite{AGR}.

\begin{defn}\label{signdec} A triple $(\mathcal A, \tT, (-))$ of the second kind
has a
 \textbf{signed decomposition}
  if  there is a submonoid $\tT^+$ of $\tT$ and a $\tT^+$-submodule $
\mathcal A ^+$ spanned by $\tTz^+ := \tT^+ \cup \{ \zero \}$,
together with injective homomorphisms $\mu _i : \mathcal A ^+ \to \mathcal A $
with $\mu_i = (-)\mu _j$, for $i, j \in \{0, 1\}$, $i\neq j$, and
$\tT^+$-module homomorphisms $\pi _i : \mathcal A \to \mathcal A ^+$ for $i =
0,1$ satisfying $\pi _0 \mu _0 = 1_ {\mathcal A ^+}$, $\pi _1 \mu _0
= \pi _0 \mu _1 = 0,$ and $\mu _0 \pi_0 (-) \mu_0 \pi _1 =
1_{\mathcal A}.$

 We identify  $\mathcal A ^+$ with $\mu _0 ( \mathcal A ^+)$ and define  $\mathcal A ^- := \mu _1 ( \mathcal A ^+)$,
both as $\tT^+$-submodules of $\mathcal A .$



%
\end{defn}

\begin{lem}\label{signd}
Let $(\mathcal A, \tT, (-))$ be a triple with a signed
decomposition.
\begin{enumerate}
 \eroman
 \item $\mathcal A ^+  + \mathcal A ^- =   \mathcal
 A.$
 \item Any element of  $\mathcal A $ is uniquely represented as $ b_0 (-) b_1$ for $b_0, b_1 \in  \mathcal A^+.$
In particular,  $  \mathcal
 A^\circ = \{ a(-)a: a \in \mathcal A ^+  \}.$
    \item
Defining $\tTm =   (-)\tTp$, we have $\tTp \cap \tTm  = \emptyset.$
 \end{enumerate}
\end{lem}
\begin{proof} (i) Any $b \in \mathcal A$ satisfies  $b = \mu _0 \pi_0 (b)(-)
\mu_0 \pi _1 (b) \in \mathcal A ^+ + \mathcal A ^-.$

 (ii) From $(i)$, any element $b \in \mathcal A$ can be written as $b=b_0(-)b_1$ for $b_0,b_1 \in \mathcal A^+$. We have
 \[
 b_0 =  \pi _0 \mu _0(b_0 (-) b_1) = \pi _0 \mu _0 b, \quad b_1 = \pi _1 \mu _1(b_0 (-) b_1)= \pi _1 \mu _1 b,
 \]
showing that the decomposition is unique. The second assertion is clear.

 (iii) If $a \in \tTp \cap \tTm ,$ then $a (-) \zero = \zero (-) a,$ so $a = \zero.$
 \end{proof}

\subsection{Systems}$ $

 We round out the structure  with a \textbf{surpassing relation}
$\preceq$ ( \cite[Definition~1.31]{Row16} and also  described in
\cite[Definition~2.11]{JuR1}).

\begin{defn}\label{precedeq07}
A \textbf{surpassing relation} on a triple $(\mathcal A, \tT, (-))$,
denoted
  $\preceq $, is a partial preorder satisfying the following, for elements $a \in \tT$ and $b_i \in \mathcal A$:

  \begin{enumerate}
 \eroman
    \item  $\zero \preceq c^\circ$
 for any $c\in \mathcal A$.
\item  If $b_1 \preceq b_2$ then $(-)b_1 \preceq (-)b_2$.
  \item If $b_1 \preceq b_2$ and $b_1' \preceq b_2'$ for $i= 1,2$ then  $b_1 + b_1' \preceq b_2
   + b_2'.$
    \item   If  $a \in \tT$ and $b_1 \preceq b_2$ then $a b_1 \preceq ab_2.$
    \item   If  $a\preceq b $ for $a,b \in \tT,$ then $a =  b.$
   \end{enumerate}


\end{defn}

The justification for these definitions is given in
\cite[Remark~1.34]{Row16}.

\begin{rem}\label{indPO}  Any surpassing relation $\preceq$ induces a partial preorder $\le$  on $\mathcal A$ given by $ a_0 \le a_1$ iff $a_0^\circ \preceq a_1^\circ .$
 This restricts to a partial preorder on $\mathcal A^\circ$, which ties in with the modulus,  as  seen in
\cite[Example~2.16]{JMR}.  \end{rem}


%

\begin{defn}\label{Nu}
For a triple $\mathcal A$ with a surpassing relation $\preceq$, we
let
\[
\mathcal A_{\textrm{Null}}:= \{ b \in \mathcal A: b\succeq \zero
\}.\,\footnote{One could also require $b^\circ = b$. This has
categorical advantages seen in \S\ref{Grand} but has algebraic
drawbacks.}\]
\end{defn}

\begin{rem} $ {\mathcal A}_{\Null}$ is a $\tT$-submodule of ${\mathcal A}$ containing $\mathcal A
    ^\circ$.
\end{rem}

\begin{lem}\label{circget7}
Let $\mathcal A$ be a triple with a surpassing relation $\preceq$. Then we have $$b \preceq b  +c,\quad \forall b \in \mathcal A, ~\forall c \succeq \zero.$$
\end{lem}
\begin{proof}
This directly follows from (iii) of Definition \ref{precedeq07}.
\end{proof}
%
%



%

\begin{prop} [{\cite[Lemma~2.11]{JMR}}]\label{nab} If $b_1\preceq b_2$, then $b_2 (-) b_1  \succeq \zero$ and  $b_1 (-) b_2 \succeq
\zero.$ In   particular, if   $b  \in {\mathcal
A}_{\operatorname{Null}}$, then  $ (-) b   \in {\mathcal
A}_{\operatorname{Null}}$.
\end{prop}

Some main cases of triples and systems are  defined as follows:

\begin{example} [{\cite[Definition~2.17]{JuR1}, \cite[Definition~1.70]{Row16}}]\label{precmain0}$ $
\begin{enumerate}\eroman
\item
Given a triple $(\mathcal A, \tT ,
(-))$, define $a \preceq_\circ c$ if $a + b^\circ = c$ for some $b
\in \mathcal A.$ Here the surpassing relation $\preceq$ is
$\preceq_\circ$, and ${\mathcal A}_{\operatorname{Null}} = \mathcal
A^\circ$.
\item
The symmetrized triple (to be considered in \S\ref{symmod}, in
particular Definition~\ref{symsyst}) is a special case of
$\mathrm{(i)}$.
\item
More generally, given a surpassing relation $\preceq$, define its restriction $\preceq_\Null$ by $a \preceq_\Null c$ if $a + b  = c$ for some
$b \in \mathcal A_\Null.$
\item (Even more generally, taken from \cite{AGR}) Given a triple  $(\mathcal A, \tT, (-))$ and a
  $\tT$- submodule $\mathcal I \supseteq \mathcal A^{\circ}$ of $\mathcal A$ closed under $(-)$, we define the  $\mathcal I$-\textbf{relation}  $ \preceq_{\mathcal I}
  $ by
    $b _1  \preceq_{\mathcal I} b_2$ if  $b _2  = b_1 +c$ for some $c \in \mathcal
    I.$ This need not be a surpassing
    relation, but is relevant to \S\ref{Grand}.
\item
Take $\preceq$ to be set inclusion when $\mathcal{A}$ is obtained
from the power set of      a hyperfield,  see   \cite[\S 3.6,
Definition~4.23]{Row16}, 
\cite[\S 10]{JuR1}; we denote it here as $\preceqh$.  ${\mathcal
A}_{\Null}$ consists of those sets containing $\zero$, which is the
version usually considered in the hypergroup literature, for
instance, \cite{BB1} and \cite{GJL}.
\end{enumerate}\end{example}

%

 \begin{defn}\label{Tsyst} $ $
 \begin{enumerate} \eroman
 \item
    A \textbf{system} (resp.~\textbf{pseudo-system}) is a quadruple $(\mathcal A, \tT_{\mathcal A} ,
 (-), \preceq),$ where $\preceq$ is a surpassing relation  on the
 triple
 (resp.~ pseudo-triple) $(\mathcal A, \tT_{\mathcal A}  , (-))$,
 which
 is \textbf{uniquely negated} in the sense that
 for any $a \in \tT_\mathcal A$,  there is a unique element $b$ of
 $\tT_\mathcal A$ for which $\zero \preceq a+b$ (namely $b =
 (-)a$)\footnote{This slightly strengthens the version of ``uniquely
    negated,'' for triples, used in \cite{Row16}, which says that there
    is a unique element $b$ of $\tT$ for which $ a+b \in \mathcal
    A^\circ.$}.
\item
The system $(\mathcal A, \tT_{\mathcal A} , (-), \preceq)$ is called
a \textbf{semiring system} when $\mathcal A$ is a semiring.
 \end{enumerate}


\end{defn}

We want to view triples and their systems as the
  ground structure over which we build our representation theory. We call this a \textbf{ground system}.
 A
range of examples of ground systems is given in
\cite[Example~2.16]{JMR}, including ``supertropical
  mathematics.''

Here is a weakened version of $\preceq$, in view of
Proposition~\ref{nab}:

 \begin{defn}  For any $b,b' \in \mathcal
A$, We say that $b$ \textbf{balances} $
    b'$  written $b \nabla b'$, if $b(-)b' \in \mathcal
A_{\textrm{Null}}$.
\end{defn}

\begin{prop}\label{nablatrop} $\nabla$ is transitive for any $(-)$-bipotent triple of height 2.
\end{prop}
\begin{proof}  We want to show that $b_0\nabla b_1 $ and $b_1\nabla b_2 $ imply $b_0\nabla b_2 $.
This is obvious if all $b_0, b_1$ are in $\tT$, since then $b_0 = b_1$. If $b_0 $ or $b_1$ is in $\mathcal A^\circ$, say $b_0$, then $b_1 \preceq b_0$. We can easily see that transitivity holds in this case. Similarly, the result follows if both  $b_0 $ and $b_1$ are in $\mathcal A^\circ$.

\end{proof}

%

\subsection{Systemic modules}$ $

We fix a ground triple $(\mathcal A, \tT , (-))$ or ground system
$(\mathcal A, \tT , (-),\preceq)$, according to context, which we
simply denote by $\mathcal{A}$ as long as there is no possible
confusion.

 \begin{defn} A  \textbf{module with a negation map} (or just called \textbf{module with negation}) over the triple $(\mathcal A, \tT , (-),\preceq)$ is an  $\mathcal A$-module with a negation map $(-)$
 satisfying
 \begin{enumerate}
\item[(a)]
 $((-)a)m = (-)(am) = a((-)m)$ for $a \in
\mathcal A, \ m \in \mathcal M.$
\item[(b)]
$ \mathcal M$ is \textbf{uniquely negated} in the sense that for any
$a \in \tT_\mathcal M$,  there is a unique element $b$ of
$\tT_\mathcal M$ for which $\zero \preceq a+b$ (namely $b = (-)a$).
\end{enumerate}

A \textbf{systemic module} $\mathcal M : = (\mathcal M,
\tT_{\mathcal M}, (-), \preceq)$ over  $(\mathcal A, \tT ,
(-),\preceq)$ is an $(\mathcal A, \tT , (-),\preceq)$-module
$\mathcal M$ together with:
\begin{enumerate}\eroman
\item
a subset $\tT_{\mathcal M}$ spanning $\mathcal M$, and satisfying $\tT
\tT_{\mathcal M}\subseteq \tT_{\mathcal M}$;
\item
a surpassing relation $\preceq_{\mathcal M}$ also  denoted   as
$\preceq$, satisfying
 $am \preceq  a'm' $ for $a \preceq  a' \in
\mathcal A, \ m  \preceq  m' \in \mathcal M.$
\end{enumerate}
\end{defn}

%

 \begin{defn}\label{signdec1} A module $\mathcal M$ with negation is $\circ$-\textbf{idempotent} if
 $(b^\circ)^\circ = b^\circ$ for each $b \in \mathcal M.$
\end{defn}

 \begin{lem} Any $(-)$-bipotent module of second kind is
 $\circ$-idempotent.
 \end{lem}

\begin{proof}  If $(-)$ is of second kind then $b^\circ + b^\circ =
 (b+b) (-) (b+b) =  b (-) b = b^\circ.$
\end{proof}

For  $(-)$  of first second kind, $b+b+b = b+b$ for supertropical
algebra, but various ``layered semialgebras''  of
\cite[Example~1.50]{Row16} are counterexamples. One can extract a
``largest''  $\circ$-idempotent submodule.

\begin{lem}  For any module $\mathcal M$ with negation, $\{ b \in  \mathcal M:  (b^\circ)^\circ = b^\circ\}$
is a
 $\circ$-idempotent submodule with negation.
 \end{lem}

\begin{proof}  If $ (b^\circ)^\circ = b^\circ$ and $ (c^\circ)^\circ =
c^\circ,$ then $ ((b+c)^\circ)^\circ =   (b^\circ)^\circ +
(c^\circ)^\circ =  b^\circ + c^\circ = (  b+ c)^\circ,$ and $((-b)^
\circ)^\circ =  (b^\circ)^\circ  = b^\circ = (-b)^ \circ.$
\end{proof}

 In \S\ref{Grand}
  $\circ$-idempotence will play a special role. Following
  \cite{JuR1},
we   consider the category of  systemic modules $(\mathcal M,
\tT_{\mathcal M} , (-), \preceq)$ over a semiring
 system  $(\mathcal A, \tT_{\mathcal A} , (-), \preceq).$

\begin{example}
 \label{reve}$ $
 \begin{enumerate} \eroman
 \item For any   module $\mathcal M $ with negation, $a \mathcal M $ is a
submodule of $\mathcal M $ with negation, for any $a \in \tT$.
\item  For a system $(\mathcal A, \tT_{\mathcal A} ,
 (-), \preceq),$ and an index set $I$, ($\mathcal A ^{(I)}, \cup _{i\in I} \tT e_i , (-), \preceq)$ is a
systemic module, where $(-)$ and $ \preceq$ are defined
componentwise, and the $e_i$ are the usual vectors with $\one$ in
the $i$ position, comprising a  \textbf{base} $\{ e_i : i \in I\}$
of $\mathcal A ^{(I)}$. It is uniquely negated, seen componentwise.

\item This is in the spirit of hyperfields. Suppose  $(\mathcal A, \tT ,
(-),\preceq)$ is a system. Then the operations extend elementwise to
the power set $\mathcal P (\mathcal A)$ and  $\mathcal P(\tT) $,
which can be viewed as a systemic module over $(\mathcal A, \tT ,
(-),\preceq)$. But now we  define $\preceq$ to be set inclusion.
\item More generally, for any  systemic module $\mathcal M : = (\mathcal M,
\tT_{\mathcal M}, (-), \preceq)$ we define  the \textbf{power set
(systemic) hypermodule} $(\mathcal P (\mathcal M),\tT_{\mathcal P
(\mathcal M)},(-),\subseteq)$, where
$\tT_{\mathcal P (\mathcal M)} = \mathcal P (\tT_{\mathcal M}),$
with elementwise operations, elementwise $(-)$, and with surpassing relation
$\subseteq$ on $\mathcal P (\mathcal M)$.
\end{enumerate}
\end{example}

\begin{defn}\label{modu218}
$\preceq$-\textbf{submodules} $\mathcal N$ of a systemic module
$\mathcal M $ are defined in the usual way, with the extra assumption
that if $b \in \mathcal N$ then $b+c \in \mathcal N$
for all $c \succeq \zero$ in
$\mathcal M.$ 
\end{defn}

In particular, any $\preceq$-submodule contains $\mathcal
M_{\operatorname{Null}}$ since $\zero \in \mathcal N$.

\begin{defn}
  An element $b\in \mathcal A$ is \textbf{null-regular over}  a $\preceq$-submodule $\mathcal N \subseteq
 \mathcal M$,   if  the following property:
 \[
 \textrm{ ``For each } y \in \mathcal M, \textrm{ there exists } z \in \mathcal{N} \textrm{ such that }\ by \succeq
 z,"
 \]
 implies $b \succeq \zero$.

 The element $b\in \mathcal A$ is \textbf{null-regular}   if $b$  is
 null-regular  for $\mathcal{N}=\zero$, i.e., if $b \mathcal M \succeq \zero$
 then $b\succeq \zero$.

 %
\end{defn}

%
%
%

\begin{example}\label{example: annihilator} \begin{enumerate}\eroman
 \item
Let $\mathcal A$ be a system and $\mathcal M$ be a systemic module.
For any subset $\mcS \subseteq \mathcal A,$ define the (module)
\textbf{annihilator}
\[
\Ann _{ \mathcal M}\, \mcS = \{ b \in   \mathcal M :  s b \in
\mathcal M_{\Null }, ~ \forall s \in \mcS\}.
\]
Then $\Ann _{ \mathcal M}\, \mcS$ is a $\preceq$-submodule of
$\mathcal M$. In fact, since $\Ann _{ \mathcal M}$ is clearly a
$\tT$-module, the only nontrivial part is to show that for $c \in
\Ann _{ \mathcal M}\, \mcS$, we have that $c+d \in \Ann _{ \mathcal
M}\, \mcS$ for any $ d \succeq \zero$. But, we have that $sc \in
\mathcal M_{\Null }$ for any $s \in \mcS$, and hence
\[
s(c+d) = sc +sd \succeq \zero +\zero=\zero,
\]
showing that $c+d \in \Ann _{ \mathcal M}\, \mcS$.

 \item In particular we have the submodule $ \Ann _{ \mathcal M}\, a = \{ b
\in   \mathcal M: ab \in \mathcal M_{\Null }\}$, for any $a \in
\tT.$

 \end{enumerate}
\end{example}

 \begin{defn}
For subsets $S_0 , S_1 \subseteq \mathcal M  $, we write  $S_0
\preceq S_1$ if for any $b_0  \in S_0$, there exists $b_1 \in S_1$
such
that $b_0 \preceq b_1$. 
 \end{defn}
%

An alternate definition to $\preceq$ on sets could be that for any
  $b_1  \in S_1$, there exists $b_0  \in S_0$ such that $b_0
\preceq b_1$. However, this would mean $S_0 \preceq \mathcal
M^\circ$ for any set $S_0 $ containing $\zero,$ which is too inclusive to be a useful criterion.



\begin{defn}\label{signedmod}
A systemic module $\mathcal M$ with a \textbf{signed decomposition},
is an $\mathcal A$-module over a system $(\mathcal A, \tT_{\mathcal
A} ,
 (-), \preceq)$ which itself has a signed
decomposition, defined to be a   union $\mathcal M = \mathcal M ^+
\cup \mathcal M ^- \cup \mathcal M ^\circ$ of pairwise disjoint
submodules, where $ \mathcal M ^+ = \mathcal A^+
  \mathcal M ^+ + \mathcal A^-   \mathcal M ^- $ and $\mathcal M
^- = \mathcal A^+ \mathcal M ^- + \mathcal A^- \mathcal M ^+.$
\end{defn}

We have the analog of Lemma~\ref{signd}:

\begin{lem}\label{signd1}
Let $(\mathcal M,\tT_{\mathcal M},(-),\preceq) $ be a  systemic
module with a signed decomposition.
\begin{enumerate}
 \eroman
 \item $\mathcal M ^+  $ and $ \mathcal M ^-$  are $ \mathcal A ^+$-submodules of $ \mathcal M$, with  $\mathcal M ^+  + \mathcal M ^- =   \mathcal M.$
 \item Any element of  $\mathcal M $ is uniquely represented as $ b_0 (-) b_1$ for $b_0, b_1 \in  \mathcal M^+.$
In particular,  $  \mathcal M^\circ = \{ b(-)b: b \in \mathcal M ^+
\}.$
 \item  $(\mathcal M^+, {\tT_{\mathcal M}^+}, (-),\preceq)$ is a  systemic
module over the system $(\mathcal A^+, \tT_{\mathcal A}^+ ,
 (-), \preceq)$.
\end{enumerate}
\end{lem}
\begin{proof} We see by Definition~\ref{signedmod} that $\mathcal M ^+  $ and $ \mathcal M ^-$
are closed under multiplication by $ \mathcal A ^+$.
 \end{proof}

This is to be made relevant via symmetrization, in
Theorem~\ref{morphsym7}.

\subsection{$\preceq$-Morphisms}\label{morph}$ $

 A key philosophical point here is that $\preceq$, which plays such
 a prominent structural role, most also enter into the other aspects
 of the theory. But then we are confronted with the thorny question,
 when do we use
 $\preceq$ and when $\succeq$? Our pragmatic response is always to take
 that direction which provides the best theorem.

We work over a ground system $\mathcal A= (\mathcal A, \tT_{\mathcal
A}, (-), \preceq)$ and consider systemic modules $(\mathcal M,
\tT_{\mathcal M}, (-), \preceq)$.
%

\begin{defn}\label{ordmp}
Let $\mathcal A= (\mathcal A, \tT_{\mathcal A}, (-), \preceq)$ be a system.
\begin{enumerate}\eroman
\item  {\cite[Definition~2.37]{JuR1}} An \textbf{ordered map} of pre-ordered modules $$f: (\mathcal M,
\tT_{\mathcal M}, (-), \le)\to (\mathcal M', \tT_{\mathcal M'} ,
(-)', \le')$$ is a function $f: \mathcal M \to \mathcal M'$
satisfying the following
properties  for $a  \in \tT$ and   $b\le b'$, $b,b'$  in $\mathcal
M$:
\begin{enumerate}\eroman
\item
$f (\zero ) = \zero  .$
\item  $ f(a  b)=  a   f( b) $.
\item $ f((-)b)=   (-)
f(b)$ (automatic from (b) since we are assuming that $(-)\one \in
\tT);$
\item $ f(b) \le ' f(b').$
\end{enumerate}
\item  {\cite[Definition~2.37]{JuR1}} A \textbf{$\preceq$-morphism} of systemic modules $$f: (\mathcal M,
\tT_{\mathcal M}, (-), \preceq)\to (\mathcal M', \tT_{\mathcal M'} ,
(-)', \preceq')$$ is  an ordered map $f: \mathcal M \to \mathcal M'$
  (taking the preorder
to be $\preceq$) also satisfying $$ f(b + b') \preceq ' f(b) + f(
b'), \ \forall b,b' \in \mathcal M.$$

\item  {\cite[Definition~2.37]{JuR1}} A $\preceq$-morphism  $$f: (\mathcal M,
\tT_{\mathcal M}, (-), \preceq)\to (\mathcal M', \tT_{\mathcal M'} ,
(-)', \preceq')$$  of systemic modules is     \textbf{systemic} if
$f( \mathcal M )$ is a systemic submodule of $ \mathcal M'$
satisfying the property that $f^{-1}(\tT_ {\mathcal M'}) \cap \tT_
{\mathcal M}\ne \emptyset;$ $f$ is \textbf{tangible} if $f(\tT_
{\mathcal M}) \subseteq \tT_ {\mathcal M'}.$

\item   A \textbf{$\succeq$-morphism} of systemic $\mathcal A$-modules   is an ordered map $f: \mathcal M \to \mathcal M'$
also satisfying $$ f(b  + b') \succeq ' f(b) + f( b'), \ \forall
b,b' \in \mathcal M.$$
\end{enumerate}
\end{defn}

Note for ordered maps that if $b\succeq \zero$ then  $ f(b)\succeq
\zero$, by (a) and (d). Furthermore, ordered maps obviously satisfy
the convexity condition that if $b \le b' \le b''$ then $f(b) \le
f(b') \le f(b'').$

%

 By a \textbf{homomorphism} we mean that
equality holds instead of $\preceq'$ in (ii). (This would correspond
to a ``strict homomorphism'' of hyperrings in \cite[Definition
2.3]{Ju}.) A \textbf{homomorphism} of semiring systems is also
required to satisfy $ f(b b')= f(b ) f( b') .$

\begin{example}\label{triv1} \begin{enumerate}\eroman
 \item
Let $\mathcal{A}$ be a system and $\mathcal {M}$ be a systemic
module over $\mathcal A$. The \textbf{null homomorphism}  $
f_{\operatorname{Null}} :\mathcal M \to \mathcal M$ is given by
$f_{\operatorname{Null}}(b )= b^\circ.$

\item The \textbf{zero homomorphism}  $
\zero _{\mathcal M } :\mathcal M \to \mathcal M$ is given by
$f_{\operatorname{Null}}(b )= \zero, \ \forall b.$

\item  An example of a
$\preceq$-morphism which is not a homomorphism. Let $(\mathcal A,
\tT, (-)),\preceq) $ be a $(-)$-bipotent system of {height} 2, i.e.,
$\mathcal A =   \tT \cup  \mathcal A^ \circ$. Define $\varphi:
\mathcal A \to \mathcal A$ by  $\varphi(a) = a^\circ$ for $a \in\tT$
and $\varphi(b) =\zero $ for $b \in \mathcal A^ \circ .$ We claim
that $\varphi$ is a $\preceq$-morphism which is not a homomorphism.
Indeed, for $a' \ne (-) a$ in $\tT$, $\varphi (a+a') = (a+a')^\circ
= a^\circ +(a')^\circ = \varphi (a) +\varphi (a'),$ but  $\varphi
(a(-)a) = \zero \ne \varphi (a)^\circ = \varphi (a)+ \varphi
((-)a);$ for $a + b^ \circ = b^\circ,$ $\varphi(a + b^ \circ)= \zero
\preceq a^\circ = \varphi(a)$ whereas for $a + b^ \circ = a,$
$\varphi(a + b^ \circ)= a^\circ = \varphi(a)+\varphi( b^ \circ).$
\end{enumerate}  \end{example}

Example~\ref{morphmean} below gives instances where we would prefer
to use $\preceq$-morphisms rather than homomorphisms.
 On the other hand,
as with \cite{CC, Ju}, to facilitate results, we often make the
stronger assumption that $f$ is a homomorphism. Ironically, although
the zero homomorphism is null, we prefer to bypass it because it is
too special. For example, the zero homomorphism is not tangible. The
following observation helps.
%

\begin{rem}\label{2.30} In the case of  signed systemic modules, any
$\preceq$-morphism $$f: (\mathcal M, \tT_{\mathcal M}, (-),
\preceq)\to (\mathcal M', \tT_{\mathcal M'} , (-)', \preceq')$$ can
be decomposed as $f = (f_0,f_1),$  where $f(b)    = f_0(b)(-)f_1(b)$
with $f_i (b) \in {\mathcal M'}^+.$ In other words, $f|_{\mathcal
 M^+} = f_0.$ and  $f|_{\mathcal
 M^-} = f_1.$ We can continue this decomposition by defining $f_{0,i}(b) = \pi_0 (f_{i}(b))$
 and $f_{1,i}(b) = \pi_1 (f_{i}(b))$ for $i = 0,1$. This is all
well-defined, in view of Lemma~\ref{signd1}.

 Conversely, given $f_i: \mathcal M \to {\mathcal M'}^+ ,$ we define
 $f= (f_0,f_1): \mathcal M \to \mathcal M'$ by $f(b)  = f_0(b)(-)f_1(b)$,
 and we continue as before. Thus any $\preceq$-morphism $f$ corresponds to positive $\preceq$-morphisms $(f_{i,j}: 0 \le i,j \le 1)$.\end{rem}

%

\subsubsection{Images}$ $

\begin{defn}
Let $\mathcal M$ and $\mathcal M'$ be systemic modules over a system $\mathcal A$, and $$f:
(\mathcal M, \tT_{\mathcal M}, (-), \preceq)  \to (\mathcal M',
\tT_{\mathcal M'}, (-), \preceq)$$ be a $\preceq$-morphism of systemic modules.
\begin{enumerate}\eroman
 \item
  We define the \textbf{module image} $f(\mathcal M )$ of $f$ in the usual way as $\{
f(b):b\in  \mathcal M \}.$
\item
The $\preceq$-\textbf{module image} of $f$ is
$$f(\mathcal M )_\preceq :=  \{ b'\in  \mathcal M' : f(b)    \preceq  b', \quad \text{for some} \quad b\in  \mathcal M
\}.$$ (Thus  $f(\mathcal M ) \preceq f(\mathcal M )_\preceq.$)

The $\succeq$-\textbf{module image} of $f$ is
$$f(\mathcal M )_\succeq :=  \{ b'\in  \mathcal M' :    f(b)      \succeq  b', \quad \text{for some} \quad b\in  \mathcal M  \}.$$

%

\item
$f$ is \textbf{null}  if $f(\mathcal
M) \subseteq \mathcal M'_{\operatorname{Null}}$.

\item
$f$ is $\preceq $-\textbf{onto} if $f( {\mathcal M} )_\preceq \,
=\mathcal M'$   i.e.,  if for all $ b' \in \mathcal M'$ there is $b
\in \mathcal M$ for which  $ f(b)\preceq b'$.

\item
$f$ is $\succeq $-\textbf{onto} if $f( {\mathcal M} )_\succeq \,
=\mathcal M'$  i.e., if for all $ b' \in \mathcal M'$ there is $b
\in \mathcal M$ for which  $  f(b ) \succeq b'$.
 \end{enumerate}
 \end{defn}

The set
 $f(\mathcal M
)_\succeq $ is a systemic submodule of $\mathcal M'$  for any
$\succeq$-morphism $f$. In fact, in this case, $f(a)+f(b) \succeq
f(a+b)$. In particular, $f(\mathcal M )_\succeq $ is closed under
addition. One can easily check the remaining conditions.

%


\subsection{The role of universal algebra and model
theory}\label{unalg}$ $


Although we do not want to get bogged down in formalism, it is
appropriate to see how all of this ties in with model theory and the
venerable theory of universal algebra, invented for general
algebraic theories. Universal algebra is defined by a ``signature''
comprised of various sets $A_1, \dots, A_t,$ called ``carriers,''
operations $\omega _{j,m}: A_{i_1} \times \dots \times A_{i_m} \to
A_{i_{m+1}}$ on the the sets, and ``universal relations,'' also
called ``identities,'' which equate evaluations of operators, i.e.,
$$p(x_1,\dots, x_t) = q(x_1,\dots, x_t),$$ where $p$ and $q$ involve
composites of various operators $\omega _{j,m}$. In the theory of
triples, we take carriers $A_1 = \mathcal A,$ $A_2 = \tT,$
operations $\omega _{0,1} = \zero,$ $\omega _{0,2} = \one$, $\omega
_{1,1} = (-),$ $\omega _{2,1}: A_1 \otimes A_1 \to A_1$ to be
addition, $\omega _{2,2}: A_2 \otimes A_1 \to A_1$ to be the
$\tT$-action, and $\omega _{2,2}: A_1 \otimes A_1 \to A_1$ to be
multiplication when $\mathcal A$ is a semiring,  and in case when
$\tT \subseteq \mathcal A$, $\omega _{2,1}: A_2 \otimes A_2 \to A_1$
to be addition on $\tT$. The universal relations we have are
associativity, distributivity (when it is considered part of the
structure), and the properties of the negation map. For example
$(-)((-)b) = b$ can be written as $\omega _{1,1}(\omega _{1,1}(b)) =
b.$ We can also describe a  module over a triple as a  universal
algebra.

 As in
\cite[\S 2]{Row16}, universal algebra provides a guide for our
definitions, especially with regard to the roles of possible
multiplication on $\tT$ and the negation map. 

So far we are missing one critical ingredient, $\preceq$. For this
purpose we incorporate a surpassing relation into the signature, and
stipulate that if $x_i \preceq x_i'$ then $$\omega _{j,m}(x_1,
\dots, x_m) \preceq \omega _{j,m}(x_1', \dots, x_m').$$   A common
example of such a modification needed in universal algebra is the
theory of ordered groups. One must be careful. The ordered group
$(\Q, +)$ fits into this approach, but the ordered group $(\Q,
\cdot)$ does not, since $-2 < -1 $ and $-5 < 1$ but $10>1.$ 

\begin{note*}
 Surpassing relations, not being identities,   play a role in  modifying universal
algebra. Having $\preceq$ at our disposal, we want it to replace
equality in the basic definitions, especially since it is relevant
to recent work in hyperfields in \cite{CC,Henry,Ju,Ju1,JuR1}.

%
%
%

Let $\AA = (\mathcal A _1, \dots \mathcal A_t)$ and $\BA= (\mathcal
A' _1, \dots \mathcal A'_t)$  be carriers of a given signature  with
surpassing relations $\preceq$ and~$\preceq'$, respectively.

Let us  formalize   $\preceq$-morphisms described above, in terms of
universal algebra with a surpassing relation. A
$\preceq$-\textbf{morphism} $f: \AA \to \BA$ is a set of maps $f_j:
\mathcal A_j\to \mathcal A_j'$, $1 \le j \le m,$ satisfying the
properties:
\begin{enumerate}\eroman \item $f(\omega (b_{1 }, \dots, b_{m }))\preceq ' \omega
(f_1(b_{1 }), \dots, f_m(b_{m })),$  for every operator $\omega :
\mathcal A_{i_1} \times \dots \times  \mathcal A_{i_m} \to \mathcal
A_{i_{m+1}},$  with $ b_j \in \mathcal A_{i_j}.$

\item  For every operator $\omega : \mathcal
A_{i_1} \times \dots \times  \mathcal A_{i_m} \to \mathcal
A_{i_{m+1}}$, if $b_{j} \preceq c_{j}$ in $\mathcal A_{i_j}$ for
each $1\le j \le m,$ then
$$\omega (f_1(b_{1}), \dots,f_m(b_{m}))\preceq '
\omega (f_1(c_{1}), \dots, f_m(c_{m})).$$
\end{enumerate}

\end{note*}

\subsection{Symmetrization}\label{symmod}$ $

We introduce the symmetrization, which is a way to obtain a negation map, when there is not
a natural one. It is an important technique, introduced by Gaubert \cite{Ga} to study
vector spaces over semifields; Here we put it in the context of triples and systems.

A $\Z _2$-graded semigroup $\widetilde{\mathcal A}$ is also called a \textbf{super-semigroup} (not to be confused
with ``supertropical''), i.e., $\widetilde{\mathcal A} = \mathcal A_0 \oplus \mathcal A_1$
as semigroups. A  \textbf{super-semiring} $\mathcal A$  is a
super-semigroup $\widetilde{\mathcal A } = \mathcal A_0 \oplus
\mathcal A_1$ that is a semiring  satisfying $\mathcal A_0^2,
\mathcal A_1^2 \subseteq \mathcal A_0$ and $\mathcal A_0 \mathcal
A_1, \mathcal A_1 \mathcal A_0 \subseteq \mathcal A_1.$
A major example is a signed decomposition, where $\mathcal A_0 =
\widetilde{\mathcal A}^+$ and $\mathcal A_1 = (-)\widetilde{\mathcal
A }^+$. We view $ \widehat{\mathcal A}:=\mathcal A \oplus \mathcal
A$ as a $\tT$-module via the diagonal action, and as a
super-semigroup, where  $\mathcal A _0,$ $\mathcal A _1 $ each is a
copy of $\mathcal A $. 
Following \cite[\S2.2]{JuR1} we impose a canonical
$\widehat{\mathcal T}$-module structure on $\widehat{\mathcal A}$ in
the following way.

\begin{defn}\label{wideh7}
 For any $\tT$-module $\mathcal A$,
the \textbf{twist action} on $\widehat{\mathcal A}$ over
$\widetilde{ \tTz }:=  \tTz \oplus \tTz$
 is given by the super-action, namely
 \begin{equation}\label{twi} (a_0,a_1)\ctw (b_0,b_1) =
 (a_0b_0 + a_1 b_1, a_0 b_1 + a_1 b_0), \ a_i \in \tTz,\ b_i \in \mathcal A.  \end{equation}
The \textbf{switch map} $(-)_{\operatorname{sw}}$ on
$\widehat{\mathcal A}$ is given by
 $(-)_{\operatorname{sw}}(b_0,b_1)= (b_1,b_0).$
\end{defn}

\begin{example}\label{semidir38} The
\textbf{symmetrized $\tT$-module} $\widetilde{\mathcal A} $ of a
$\tT$-module $\mathcal A$ containing $\tT$  is defined as the
submodule of $\widehat{\mathcal A}  $ generated by
$\{(\zero,\zero)\}$ and $\tT_{\widetilde{\mathcal A}}  := (\tT
\oplus {\zero}) \cup (\zero \oplus {\tT }).$
\end{example}

\begin{rem}\label{exp} The switch map $(-)$ on $\widehat{\mathcal A}$ is a
negation map, and all quasi-zeros have the form $(b,b)$ since
$(b_0,b_1)(-)(b_0,b_1)= (b_0+b_1,b_0+b_1).$ Also notice that
$\mathcal T_{\widehat{\mathcal A}}\cap {\widehat{\mathcal A}}^\circ
= \emptyset.$  Hence, $(\widehat{\mathcal A}, \mathcal
T_{\widehat{\mathcal A}},(-))$ is a triple for any $\tT$-module
$\mathcal A$ generated by $\tT$. Thinking of $(b_0,b_1)$ as a replacement for
as $b_0 - b_1$, we see that $(b_1,b_0)$ corresponds to $b_1-b_0 =
-(b_0 - b_1).$ When $\mathcal A$ is additively idempotent, so is
$\widehat{\mathcal A}$.
\end{rem}

In  this situation, Example~\ref{precmain0}$~(\mathrm{i})$ becomes:

\begin{defn}\label{symrprec} $(b_0,b_1)\preceq_{\operatorname {sym}} (b_0',b_1') $ if
there is $c \in \mathcal A$ such that $b_i+c = b_i'$ for $i = 0,1$.
\end{defn}

$\preceq_{\operatorname {sym}}$, the main surpassing relation used
in this paper, is $\preceq_{\operatorname {\circ}} $ with respect to
the switch map (viewed as a negation map).

 \begin{defn}\label{symsyst}
 Let $\mathcal A$ be a system. The \textbf{symmetrized system} (of $\mathcal A$) is $(\widehat{\mathcal A}, \tT_ {\widehat{\mathcal A}}, (-),
\preceq_{\operatorname {sym}})$.
\end{defn}

 \begin{thm}\label{morphsym} $ $

 \begin{enumerate}\eroman
 \item For any $\tT$-module $\mathcal A$, we can embed  $\mathcal
 A$ into $\widehat {\mathcal A}$ via $b \mapsto (b,\zero)$, thereby
 obtaining a faithful functor from the category of semirings into  the category of
 semirings with a negation map (and preserving additive
 idempotence).  $(\widehat{\mathcal A}, \tT_ {\widehat{\mathcal A}}, (-),
\preceq_{\operatorname {sym}})$ is a system when  $\tT$ generates
$\mathcal A$.

 \item The  symmetrized system $(\widehat{\mathcal A}, \tT_ {\widehat{\mathcal A}}, (-),
\preceq_{\operatorname {sym}})$ has a signed decomposition with
$\widehat{\mathcal A}^+ = \mathcal A \oplus \{ \zero \},$ the image
of $\mathcal A$ in (i), and $\widehat{\mathcal A}^- = \{ \zero \}
\oplus \mathcal A.$ Thus the functor in (i) goes to the category of
triples with a signed decomposition, and the projection to the
positive part is a retract.

 \item Any $\mathcal A$-module $\mathcal M$ yields a $\widehat{\mathcal A}$-module $\widehat{\mathcal M}= \mathcal M \oplus \mathcal M,$
 which has a signed
 decomposition   where $\mathcal M^+$ is the
 first component.

  \item For any  $b_0, b_1 \in \mathcal M$ and $S \subseteq
\mathcal A,$ and $b = (b_0, b_1),$ the set
 $$[b:S]  := \{ c \in \widehat{\mathcal M}: S c \succeq   b_0 (-) b_1\}$$
is a $\preceq$-submodule of $\widehat{\mathcal M}$.
\end{enumerate}
\end{thm}
\begin{proof}
(i) The twist multiplication matches the usual multiplication on the
first  component (where the multiplicative unit is $(1,\zero))$, and
then any $\preceq$-morphism $f$ is sent to the $\preceq$-morphism
$(f,\zero))$   on $\widehat{\mathcal A}$.

The proofs of (ii), (iii), and (iv) are straightforward.
\end{proof}

Remark~\ref{2.30} also applies in the symmetrized case $\widehat{\mathcal M}$ of an
arbitrary module $\mathcal M$, and the positive part of $\widehat{\mathcal M}$, which is the original module $\mathcal M.$

The one downside here is that the symmetrized triple of $\mathcal A
$ is not $(-)$-bipotent. The following modification, which is
$(-)$-bipotent and thus more amenable to the systemic theory and
tropical mathematics, was introduced in Gaubert's dissertation
\cite{Ga} and \cite[Proposition~5.1]{AGG1}, and explored further
under the name of ``symmetrized max-plus semiring'' in
\cite[Proposition-Definition 2.12]{AGG1}; also see
\cite[Example~1.42]{AGR}.

\begin{example}\label{surpsym} One starts with an ordered semigroup $\tT$ and a $\tT$-module $\mathcal A := \tTz =  \tT
\cup \{ \zero\}, $ and defines
$$\widetilde{\tT} := (\tT \oplus \{ \zero\}) \cup ( \{
\zero\} \oplus \tT) ,\quad \tT_{\operatorname{sym}}:=
\widetilde{\tT} \cup D \ \textrm{where } D=\{ (a,a): a\in \tT_\zero
\} \subset \tTz \oplus \tTz.$$ Thus, viewing $\tT_\zero$ as a
bipotent semiring, addition on $\tT_{\operatorname{sym}}$ also is
according to components on $\tT_\zero \oplus \{ \zero\}$, $\{
\zero\} \oplus \tT_\zero$, and $ \{ (a,a): a\in \tT_\zero \}$,
whereas ``mixed'' addition satisfies:
 $$(a_0,\zero)   + (\zero, a_1) = \begin{cases} (a_0,\zero)  \text{ if } a_0 > a_1;
 \\ (\zero, a_1) \text{ if } a_0 < a_1;  \\  (a_1, a_1)
 \text{ if }   a_0 =  a_1 ; \end{cases}$$
 $$(a_0,\zero)   + (a_1, a_1) = \begin{cases} (a_0,\zero)  \text{ if } a_0 > a_1;
 \\ (a_1, a_1) \text{ if } a_0 \le a_1;  \end{cases}$$
$$(\zero, a_0)   + (a_1, a_1) = \begin{cases} (\zero, a_0)    \text{ if } a_0 > a_1;
 \\ (a_1, a_1) \text{ if } a_0 \le a_1.  \end{cases}$$

 Multiplication in $\tT_{\operatorname{sym}}$ is the twist action as
in Definition~\ref{wideh7}, but taken  with respect to this
addition. Then $(\tT_{\operatorname{sym}},\widetilde{\tT}  ,(-))$ is
a $(-)$-bipotent triple. 

$\widehat{\tT}  $ is ordered, when we take the elements of the first
component to be positive, and thus greater than the elements of the
second component. (But this is only preserved under multiplication
by positive elements.)
\end{example}

 \begin{thm}\label{morphsym7} The assertion of Theorem~\ref{morphsym} also holds for Example~\ref{surpsym},
 providing a faithful functor from ordered semigroups to  $(-)$-bipotent triples with signed decompositions.
\end{thm}
\begin{proof}
The same verifications as in Theorem~\ref{morphsym}, with the
obvious adjustment for addition.
\end{proof}
\subsubsection{Application to \cite{CC2}}$ $

 Symmetrization   provides the link to \cite{CC2} in
obtaining homology for modules over semirings.

\begin{rem}\label{morphCC}
Connes and Consani \cite{CC2} work with the category  \BMod\ of
modules over the Boolean semifield $\mathbb B = \{ \zero, \one \}$
(where $ \one +   \one =  \one$).
 Its symmetrization then is $\widehat {\mathbb
B}_{(-)}= \{ (\zero,\zero), (\one, \zero), (\zero, \one), (\one,
\one)\},$ where $(-) (\one, \zero) = (\zero, \one),$ and the
quasi-zeros are the diagonal elements. Symmetrization decomposes
into the first and second components as modules over the Boolean
semiring.

There is a homomorphism from $\widehat{\mathbb B}_{(-)}$ onto the
supertropical extension $\mathcal A = \mathbb B \cup \mathbb B^\nu$
of~$\mathbb B$, where $(a,\zero), (\zero,a) \mapsto a$ and $(a,a)$
is fixed. The twist action (Definition~\ref{wideh7}) on
$\widehat{\mathcal A}$ over $\widehat {\tT} $ is studied in \cite[\S
4]{CC2} over $\widehat{\mathbb B}_{(-)}$ under the name of the
Kleisli category. Thus \cite{CC2}, which studies modules over
$\widehat{\mathbb B}_{(-)}$, can be viewed in this context. The
approach in \cite{CC2} is more categorical, and the reader is
welcome to interpret the results of~\cite[\S~3]{CC2} concerning
coequalizers and comonads for systemic modules. Our negation map is
similar to the Eilenberg-Moore notion of involution treated in
\cite[\S~5]{CC2}. \end{rem}

\subsection{Major examples}$ $

Let us summarize the discussion of triples and systems via examples.
More details are provided in \cite{AGR2}.

\begin{example}\label{precmain}$ $ $ $ \begin{enumerate}\eroman\item
Given a triple $(\mathcal A, \tT  , (-))$, take the surpassing
relation $\preceq$ to be $\preceq_\circ$ of Example~\ref{precmain0}.

\item  The  $\mathcal  I$-relation of Example~\ref{precmain0}(iv) provides a rather wide ranging
generalization of (i); sometimes it is better to define $b \preceq
b'$ when $b' = b^\circ$ or $b' = b+c$ for some $c \in \mathcal I,$
in order to take quasi-zeros into account.

\item The set-up
of supertropical mathematics \cite{zur05TropicalAlgebra,IR} is a
special case of (i), where $\mathcal A = \tTz \cup \tG$ is the
supertropical semiring, $(-)$ is the identity, $\circ$ is the
``ghost map,''  $ \tG =\mathcal A^\circ  $, and $\preceq$ is ``ghost
surpasses.'' Another way of saying this is that $a + a' \in
\{a,a'\}$ for $a \ne a' \in \tT$, and $a + a = a ^\circ.$ Tropical
mathematics is encoded in $\tG$, which (excluding $\zero$) often is
an ordered group, and can be viewed for example as the target of the
Puiseux valuation on the field of Puiseux series (tropicalization).

\item The
symmetrized triple can be made into  a system viewed as special case
of (i), which includes idempotent mathematics, as  explained in
 Theorem~\ref{morphsym}.

\item For a hypergroup $\tT$,
${\mathcal A}$ is the subset of the power set $\mathcal P(\tT)$
generated by $\tT$, where $(-)$ is induced from hypergroup negation,
and $\preceq$ is set inclusion. We call this a \textbf{hypersystem}.
${\mathcal A}_{\operatorname{Null}}$ consists of those sets
containing $\zero$, which  in the hypergroup literature is the set
of hyperzeros. The ``tropical'' hypergroup and sign hypergroup
actually are isomorphic to supertropical algebras, so (ii) is
applicable in these cases.

\item The
fuzzy ring of \cite{Dr} is a special case of (i), and provides
valuated matroids as explained in \cite{BB1} and
\cite[\S1.3.6]{AGR}.

\item Tracts, introduced recently in \cite{BB1}, are mostly special cases of systems, where $\tT$ is the given Abelian group $ G$, $\mathcal A
= \mathbb N[G],$ $\vep = (-) \one,$ and $N_G $ is ${\mathcal
A}_{\operatorname{Null}},$ usually taken to be $\mathcal A^\circ$.

\item (The connection to idempotent algebra) In any semiring, one has
the Green  relation given by $a \le b$ iff $a+b = b$,
\cite[Example~2.60(i)]{Row16}. Conversely, any ordered monoid with
$\zero$ gives rise to an idempotent (in fact, bipotent) semiring by
putting $a+b = b$ whenever $a\le b.$

 The only natural
negation map would be the identity, and one gets a pseudo-triple by
taking~$\tT$ to be a generating set of $\mathcal A$. But then every
element $a = a+a$ is a quasi-zero, and ${\mathcal
A}_{\operatorname{Null}}= \mathcal A$, so one needs to apply
symmetrization in order to obtain structure theory along the lines
of systems.

%
\end{enumerate}

\item
Semple and Whittle \cite{SW} implemented the notion of partial
fields, which could be seen as a weakened definition of hypergroup,
where
 addition $\tT \times \tT \to \tT$ is only defined on certain sets
 of elements. Thus $\mathcal A$ is a $\tT$-set, but not necessarily a $\tT$-module. Following an idea of \cite[\S 2.2]{BL1}, we can define
  \textbf{partial addition} on   $\tTz$ in a triple $(\mathcal A, \tT, (-))$
by accepting the operation $a_1 + a_2 = a_3$ only when $ a_3 \in
\mathcal A ^\circ.$ $\tT$~together with its  partial addition  is
called a \textbf{partial semigroup}. 
In a metatangible triple the only defined sums in the partial addition are for $a_2 = (-)a_1,$ so the partial semigroup is trivial.
\end{example}

\begin{example}\label{morphmean}
We describe $\preceq$-morphisms for the systems from Example~\ref{precmain}. These are often preferable to homomorphisms.
\begin{enumerate}\eroman\item  In supertropical
mathematics, a $\preceq$-morphism $f$ satisfies
\begin{equation}\label{mor1}
f(b_1+b_2)+\text{ghost} = f(b_1)+f(b_2);
\end{equation}
\eqref{mor1} implies that either  $f(b_1+b_2)  = f(b_1)+f(b_2)$, or
$f(b_1)+f(b_2)$ is ghost.

\item
For hypersystems, a  $\preceq$-morphism $f$ satisfies
\begin{equation}\label{morphmean3}
f(b_1\boxplus b_2)\subseteq f(b_1)\boxplus f(b_2),
\end{equation}
the definition used in \cite[Definition~2.1]{CC} and
\cite[Definition~2.4]{GJL}.  This is intuitive when $f$ maps the
hyperring $\tT$ into itself.

Given  a hypersystem $({\mathcal A}\subseteq \mathcal P(\tT), \tT,
(-),\subseteq)$ and a  hypergroup morphism $f$ over $\tT$, it is
natural to extend
 $f$ to ${\mathcal A}$ via $$f( b ) = \{ f(a): a \in
 b \}.$$
In this  case, if $f(b) (-) f(b') \succeq \zero,$ there is some
hypergroup element $a \in f(b) \cap f(b').$


\item
For fuzzy rings, in \cite[\S~1]{Dr}, also see \cite[Definition~2.17
]{GJL}, a homomorphism $$f: (K;+,\cdot, \vep _K,K_0)\to (L;+;\cdot ,
\vep _L,L_0)$$ of fuzzy rings is defined as satisfying: For any
$\{a_1 ,\dots, a_n \} \in  K ^\times$  if $\sum _{i=1}^n a_i \in
K_0$ then $\sum _{i=1}^n f (a_i) \in L_0.$ Any $\succeq$-morphism in
our setting is a fuzzy homomorphism since $L_0$ is an ideal, and
thus $\sum _{i=1}^n f (a_i) \in f (\sum _{i=1}^n a_i) + L_0 = L_0.$
The other direction might not hold. The same reasoning holds for
tracts of \cite{BB1}.

One subtle point about fuzzy rings is that, as pointed out in
\cite[Theorem 11.8]{Row16}, although any fuzzy ring gives rise to a
triple and thus a system, conceivably the set $\mathcal A_0$ could
properly contain $\mathcal A^{\circ}$,
cf.~\cite[Remark~11.3]{Row16}, and thus we could not define
$\preceq$ according to $\mathcal A_0$. (One could define $\preceq$
according to Example~\ref{precmain0}, but $\mathcal A_0$ might not
match the set of quasi-zeros.) We refer the readers to \cite{GJL} for more details on the functorial relation between hyperrings and fuzzy rings. 
%
\item Another interesting example comes from valuation theory.
In \cite[Definition~8.8(ii)]{Row16}, valuations are displayed as
$\preceq$-morphisms of semirings, writing the target of the
valuation as a semiring (using multiplicative notation instead of
additive notation) via Green's relation of
Example~\ref{precmain}(viii).   Here $\varphi(b_1 b_2) =
\varphi(b_1) \varphi(b_2) .$ If we instead wrote $\varphi(b_1 b_2)
\preceq \varphi(b_1) \varphi(b_2) ,$ we would have a
``quasi-valuation.'' This pertains in particular to the Puiseux
valuation.
\end{enumerate}\end{example}

\section{Semiring analogs of classical module theory}\label{congr0}

In this section we introduce the analogs of classical module
concepts, for modules over semiring
 systems.  We   consider the category of  systemic modules
$(\mathcal M, \tT_{\mathcal M} , (-), \preceq)$ over a semiring
 system  $(\mathcal A, \tT_{\mathcal A} , (-), \preceq).$

\subsection{Congruences, precongruences, and weak congruences}\label{congr}$ $

One of the the big differences between homomorphisms in semiring theory
in comparison to ring theory is the lack of a bijection between modules and kernels of homomorphisms.
This makes necessary the use of congruences, which are equivalence relations
that respect the given operations (in our case multiplication, addition, and negation).
Congruences are subsets of $\widehat{\mathcal A}$ which also are subalgebras.
Moreover, failure of the surpassing relation to be an identity hampers the definition of quotient (factor) objects.
Here we impose some restrictive conditions in order to overcome this problem.


\begin{defn}\label{precongmor0} A congruence $\Cong$ on a system is \textbf{convex} if
whenever $b \preceq b'$ and $(b,c) \in \Cong$ and   $(b',c') \in
\Cong$ then $c \preceq c'$.\end{defn}

\begin{lem}\label{lemma: quotient module} \begin{enumerate}\eroman
\item If $\mathcal{M}$ is a module with negation $(-)$ over a triple
and $\Cong$ is a congruence relation on~$\mathcal{M}$, then the
quotient $\mathcal{M}/\Cong$ also has a negation given by $(-)[b] =[(-)b]$.

\item Let $\mathcal{M}$ be a systemic module and $\Cong$ be a congruence
relation on $\mathcal{M}$; then the quotient $\mathcal{M}/\Cong$ is also a
systemic module with the surpassing relation given  for $[b],[b']
\in \mathcal{M} /\Cong$ by:
\begin{equation}\label{eq: quotient partial order}
[b] \preceq [b'] \iff c \preceq c' \textrm{ in $\mathcal{M}$ for
some }c \in [b], c' \in [b'],
\end{equation}
where $[b]$ is the equivalence class of $b \in \mathcal{M}$ in
$\mathcal{M}/\Cong$.\end{enumerate}
\end{lem}
\begin{proof} (i) If $[b'] = [b]$ then $(b,b') \in \Cong$ so
$(-)(b,b') = ((-)b,(-)b')\in \Cong.$

(ii) It is clear that $\mathcal{M}/\Cong$ is a module. Furthermore,
it is well-known that the relation on a quotient of a poset
defined as in \eqref{eq: quotient partial order} is again a partial
order. Now, one can easily check that $\mathcal{M}/\Cong$ is indeed
a systemic module.
\end{proof}

Lemma~\ref{lemma: quotient module} points to an important difference
between triples and systems. Whereas $\preceq$ plays a critical
model-theoretic role, it hampers the category theory since it
restricts the use of quotient objects. (The quotient of a system
need not be a system, cf.~ Example \ref{anni} below.) This will be
reflected in \S\ref{Grand}.

We generalize congruences slightly, since reflexivity often fails.

\begin{defn}\label{precongmor}
Let $(\mathcal A,\tT,(-),\preceq)$ be a system.
\begin{enumerate}\eroman
\item  The \textbf{diagonal congruence} $\Diag_{\mathcal M}$ of a systemic
 $\mathcal A$-module $\mathcal M$ is $\{(b,b): b \in \mathcal M\}$.
\item
A \textbf{precongruence} of a systemic module $\mathcal M$ is a
$\preceq$-submodule $\Cong$ of $\widehat{\mathcal M}$, which also is symmetric. In other words, precongruences satisfy all the properties of congruences except perhaps reflexivity and transitivity, and preserve the $\tT$-action.

\item $\Cong'$ is a \textbf{weak congruence} if $\Cong' = \Cong \cup
\Diag_{\mathcal A},$ where $\Cong $ is a transitive precongruence.
\end{enumerate}\end{defn}

\begin{lem} Any weak congruence is a transitive precongruence.
\end{lem} \begin{proof} The adjunction of $\Diag_{\mathcal A}$ does
not affect symmetry or transitivity.
\end{proof}

 \begin{rem} For any subset $S$ of $\mathcal A \times \mathcal A,$
 the intersection of all congruences containing $S$ is the
 ``smallest'' congruence containing $S$. But this can be much larger
 than $S$, even when $S$ is a weak congruence.
\end{rem}
 Let us see some examples.

\begin{example}\label{anni}
Let $\mathcal M$ be a  module over $\mathcal A$.
\begin{enumerate}\eroman

\item For any subset $\mcS \subseteq \mathcal A,$ the \textbf{congruence annihilator}
\[
\Ann _{\mathcal M, \operatorname{cong}}\mcS= \{ (b,b') \in \widehat{
\mathcal M}:  s b  = s  b'  , ~ \forall s \in \mcS\}
\]
is a congruence of $\mathcal M$. In particular,
$$\Ann _{\mathcal M, \operatorname{cong}}a =  \{ (b,b') \in \widehat{\mathcal M}: ab = ab'\}.$$
However, the congruence $ \Ann _{\mathcal M, \operatorname{cong}}$
need not be convex.
\item
Suppose that $\mathcal M$  has a signed decomposition over a triple
$\mathcal A$. For any $b_0, b_1 \in \mathcal M^+,$ define
$$\widehat{\mathcal A} (b_0, b_1)   = \{ (c_0 b_0 +  c_1 b_1, c_0 b_1+ c_1 b_0 ) : c_0,c_1
\in \mathcal A ^ + \}.$$ Then the   set
 $\widehat{\mathcal A} (b_0, b_1)$
is a precongruence.



\item The \textbf{$\preceq$-diagonal precongruence} is $\{(b,b') \in \widehat{ \mathcal M}: b'(-)b \succeq \zero\}.$


\end{enumerate}
\end{example}
%



\begin{defn} If $\Cong$ is a precongruence of $ \mathcal M
$, and $\Cong'$ is a  congruence,   then the \textbf{quotient
precongruence}
 $\Cong/\Cong'$ is the precongruence on  $\mathcal M/\Cong'$ consisting
 of all pairs of equivalence classes $  ([b_0],[b_1]) $ with respect to $\Cong',$
  with $b_i \in \mathcal M  $ and $\ (b_0, b_1) \in \Cong$.
\end{defn}


We can act directly on the precongruences.

 \begin{rem}  $ $
 \begin{enumerate}\eroman
\item If $\Cong' \subseteq \Cong $ are  congruences, the quotient  $\Cong/\Cong'$ is
a congruence. When  $\Cong = \Cong'$ and  $[b_0] = [b_1]$ with
respect to $\Cong,$ then   $[b_0] = [b_1]$ with respect to $\Cong',$
so $\Cong/\Cong'$ is the
 diagonal congruence.

\item When $\Cong$ is merely a precongruence, the quotient  $\Cong/\Cong'$ is   a
precongruence.
\end{enumerate}
\end{rem}

\subsubsection{Morphisms of $\preceq$-precongruences and congruences}$ $

%
%

 We consider the category of  systemic modules
$(\mathcal M, \tT_{\mathcal M} , (-), \preceq)$ over a semiring
 system  $(\mathcal A, \tT_{\mathcal A} , (-), \preceq).$




\begin{defn}\label{congmor}
Let $\Cong$ (resp.~$\Cong'$) be a precongruence on $\mathcal M$ (resp.~$\mathcal N$).

\begin{enumerate}\eroman
\item
A \textbf{precongruence
    $\preceq$-morphism} $\hat f = (f_0,f_1): \Cong \to \Cong' $ is
    given by $$f(b,b') = (f_{00}(b) +  f_{01}(b'), f_{10}(b) +
    f_{11}(b')), \ \forall b,b' \in \mathcal M$$ where $f_i: b
    \mapsto (f_{i0}(b),f_{i1}(b))$ for $\preceq$- morphisms $f_{ij} : \mathcal M \to \mathcal N .$

     When $\hat{f}$ also is a homomorphism of
modules, we call it a precongruence homomorphism.
\end{enumerate}
\end{defn}

Of course we can delete ``pre'' when   $ \Cong ,\Cong' $ are congruences.
 Clearly the
composition $\hat f \hat g$ of  precongruence $\preceq$-morphisms
(resp.~homomorphisms) is the  precongruence $\preceq$-morphism
(resp.~homomorphism) $\widehat{fg}$.

\begin{defn} $ $
\begin{enumerate} \eroman
    \item
    A precongruence $\preceq$-morphism $\hat f=(f_0,f_1) : \Cong \to
    \Cong'$ is $(\preceq_{\Cong''})$-\textbf{trivial} for a
    sub-$\preceq$-precongruence $\Cong''$ of~$\Cong'$, if $\hat f(\Cong)
    \succeq \Cong''$.
    \item
A precongruence $\preceq$-morphism $\hat f : \Cong \to \Cong'$ is
\textbf{$\preceq_{\operatorname {sym}}$-trivial} if $\hat f$ is
$\Cong''$-trivial for $\Cong''=\Diag_{\Cong'}$, i.e., $\hat f(\Cong)
\succeq \Diag_{\Cong'}$.
\end{enumerate}
\end{defn}

\begin{example}\label{triv11} The  \textbf{left $\preceq$-trivial  precongruence
homomorphism $\hat f : \Cong \to \Diag_{\mathcal M}$ is given by
$\hat f(a_0,a_1) = (a_0,a_0).$}
 It is a  precongruence homomorphism since
 $$\hat f((a_0,a_1)+ (a_0',a_1')) = \hat f(a_0+a_0',a_1+ a_1')= (a_0+a_0',a_0+ a_0')= (a_0,a_0)+ (a_0',a_0') =
  \hat f(a_0,a_1)+ \hat f(a_0',a_1').$$
 \end{example}

\subsection{Monics, kernels, cokernels, and images} $ $

\begin{defn}\label{ont}
Let $\mathcal M$ and $\mathcal N$ be systemic modules over a semiring system $\mathcal A$, and $f: \mathcal M\to \mathcal N$ be a $\preceq$-morphism.
\begin{enumerate} \eroman
\item
We say that $f$ is $\preceq$-\textbf{monic} if $f(b_0)\preceq f(b_1)
$ implies $b_0 \preceq b_1.$
\item
 We say that $f$ is \textbf{N-monic} if the symmetrization of $f$, that is $\hat{f}=(f,f)$, is
$\preceq$-monic; i.e., $f(b_0)+c = f(b_1)$ and $f(b_0') +c=
    f(b_1')$ for $c \in \mathcal N$ implies $b_0 +c' = b_1$ and $b_0' +c'
    = b_1'$ for some $c' \in \mathcal M.$

    \item $f$ is an $\preceq$-\textbf{isomorphism} if there is a $\preceq$-morphism $g: \mathcal N\to \mathcal M$
    for which $gf = 1_{ \mathcal M}$ and $ 1_{ \mathcal N} \preceq
    fg$ (elementwise).
\end{enumerate}
\end{defn}



This has been done in general in universal algebra by means of
``equalizers'' \cite[\S~9.1]{BaW} and ``kernel pairs''
\cite[Exercise~9.4.4(5)]{BaW}, but the theory is rather intricate,
so we   just bypass it.



 %

\subsubsection{Kernels}$ $

Kernels are  tricky. Let us start with a naive definition,
%

%

 \begin{defn}\label{modker}
Let $\mathcal M$ and $\mathcal N$ be systemic modules over $\mathcal
A$, and $f: \mathcal M \to \mathcal N$ a $\preceq$-morphism.
\begin{enumerate} \eroman
    \item
        A $\preceq$-morphism $f: \mathcal M\to \mathcal N$ is
    \textbf{null} if $f(\mathcal M)\subseteq \mathcal N_{\operatorname{Null}}.$

 \item A $\preceq$-submodule $\mathcal M'$ of $\mathcal M$ is   \textbf{$f$-null} if
$f(a) \in \mathcal N_{\operatorname{Null}}$ for all $a \in \mathcal
M'.$ The \textbf{null-module kernel} $\ker_{\Mod,\mathcal M} f$  of
$f$ is the sum of all $f$-null $\preceq$-submodules of $\mathcal M$. 

\item
Suppose further that $f$ is a homomorphism. The \textbf{classical
kernel}  is $\{ b  \in \mathcal M : f(b ) = \zero\}.$\end{enumerate}
 \end{defn}

We bypass the classical kernel in tropical homology, since $\zero$
is too special, cf.~ \cite[Example~3.4]{CC2}.

\begin{example}
    Define the \textbf{left multiplication map} $\ell_a : \mathcal M \to
a \mathcal M$ by  $\ell_a(b) = ab$. $\ell_a $ is a homomorphism, and
its kernel is $\Ann _{\mathcal M, \operatorname{cong}} a .$
\end{example}

 \begin{rem}$ $
 \begin{enumerate}\eroman
\item
  The composition of a null $\preceq$-morphism with any
  $\preceq$-morphism is null. Thus null $\preceq$-morphisms  take the place of the $\zero$ morphism.
\item
Let $f:\mathcal M \to \mathcal N$ be a $\preceq$-morphism. If $b \in \mathcal M_{\operatorname{Null}}$ then $f(b)\notin \tT_{\mathcal N}$ since $\mathcal T_{\mathcal N} \cap \mathcal N_{\operatorname{Null}} =\emptyset$.

%
%

\end{enumerate}
      \end{rem}
%
%
%

%
%

\subsubsection{Congruence kernels and congruence cokernels}$ $%

Here are other candidates for kernel, which may be  more natural
when working directly within the module category. Since morphisms
are defined in terms of congruences, we are led to bring congruences
into the picture.

%
%
%
%
%
%
%

 \begin{defn}
 Let $f:\mathcal M \to \mathcal M'$ be a $\preceq$-morphism of systemic modules.
 \begin{enumerate}\eroman
    \item
 The  \textbf{precongruence kernel} $\pker
 _{\operatorname{cong} }f$ of $f$ is $$\{(b_0,b_1) \in \mathcal M\oplus \mathcal M:
 f(b_0) \nabla f(b_1)\}.$$
 \item
The $\preceq$-\textbf{precongruence kernel} $\pker
_{\preceq,\operatorname{cong} }f$ of $f$ is $$\{(b_0,b_1) \in \mathcal M\oplus
\mathcal M: f(b_0) \preceq f(b_1)\}.$$

The N-\textbf{congruence kernel} $\ker  _{N}f$ \footnote{Note that this is just an equalizer of $f$.} of $f$ is
$$\{(b_0,b_1) \in \mathcal M\oplus \mathcal M: f(b_0) = f(b_1)\}.$$
 \end{enumerate}
 \end{defn}



\begin{lem}
The precongruence kernel of a homomorphism $f$ is a reflexive and
symmetric precongruence.
\end{lem}
\begin{proof}
By the definition, it is clear that the precongruence kernel is
symmetric and closed under coordinate-wise negation map and scalar
multiplication. If  $ f(b_0) $ balances $f(b_1) $ and $ f(b_0') $
balances $f(b_1') ,$ then since $f$ is a homomorphism, $ f(b_0+
b_0') $ balances $f(b_1 + b_1').$ This shows that the precongruence
kernel is closed under addition and hence it is a precongruence.
\end{proof}

But it need not be transitive!


%

\begin{rem}
The N-congruence kernel is the null-module kernel in the
symmetrization $\hat f: \widehat{\mathcal M} \to   \widehat{\mathcal
M'}$. In fact, suppose that $\widehat{f}(a,b)$ is null, i.e.,
$f(a)=f(b)$, showing that $(a,b)$ is in the $N$-congruence kernel
of~$\widehat{f}$. Conversely, suppose that $(a,b)$ is an element of
the $N$-congruence kernel of $f$. That is, $f(a)=f(b)$, showing that
$(a,b)$ is an element of the null-module kernel of $\widehat{f}$.

\end{rem}

\begin{lem}\label{kercok0}
For any  homomorphism $f: \mathcal M \to \mathcal N$, there is an
injection  $\bar f :\mathcal M/\NTkerC  f \overset{\overline{f}} \to
\mathcal N,$ given by  $\overline{f} ( \bar b  ) = f(b)$.
\end{lem}
\begin{proof} If $( b , {b' }) \in \NTkerC f,$ then $f(b) = f(b')$,
showing that $\overline{f}$ is well-defined. Clearly $\overline{f}$
is a homomorphism since $\NTkerC (f)$ is a congruence. Finally, one
can easily check that $\overline{f}$ is an injection.
\end{proof}

\begin{lem} \label{kercok} Any homomorphism  $ f: \mathcal M \to \mathcal N $
is composed as  $ \mathcal M  \to \mathcal M/\NTkerC  f \overset
{\overline{f}} \to  \mathcal N,$ where  the first map is the
canonical   homomorphism, and the second map is the injection given
in Lemma~\ref{kercok0}.
\end{lem}
\begin{proof}
The first map sends $b  \in \mathcal M$ to $\bar b$, which is
clearly a homomorphism since $\ker  _{N} f$ is a congruence. The
second map follows from Lemma~\ref{kercok0}.
\end{proof}
\subsubsection{Precongruences to $\preceq$-submodules} $ $

Since the theory mixes congruences and $\preceq$-submodules, the
next observation could be useful.

 \begin{defn}\label{congsupp21}
For a $\preceq$-precongruence $\Cong$ on a systemic module $\mathcal
M$, define $$\mathcal N  _\Cong =  \{b_0 (-)b_1: (b_0,b_1) \in \Cong
\}.$$
\end{defn}

\begin{lem}\label{congsupp23} 
$\mathcal N  _\Cong $  is a $\preceq$-submodule of  $\mathcal M$,
and gives rise to a homomorphism $\psi_\Cong: \Cong \to \mathcal N
_\Cong $ given by $\psi_\Cong (b_0,b_1) = b_0 (-)b_1.$
 \end{lem}
\begin{proof}
 {If  $b_0 (-)b_1,\  b_0' (-)b_1' \in \mathcal N  _\Cong $, where we
have  $(b_0,b_1),(b_0',b_1')\in   \Cong,$  then
$(b_0+b_0',b_1+b_1')\in   \Cong,$ so  $$(b_0(-)b_1) +(b_0'(-)b_1')=
(b_0+b_0')(-)(b_1+b_1')  \in \mathcal N  _\Cong. $$ Also $a(b_0,b_1)
= ab_0 (-)ab_1,$ whereas $(ab_0 ,ab_1) \in \Cong.$ This shows that
$\mathcal N_\Cong$ is closed under addition and scalar multiplication.
Furthermore, since $((-)b_0,(-)b_1) \in \Cong$, we have that
\begin{equation} \label{strict}
(-)b_0 (-) ((-)b_1) = (-)b_0 + b_1 = b_1(-)b_0 \in \mathcal N _\Cong,
\end{equation}
showing that $\mathcal N_\Cong$ is closed under negation. Next,
suppose that $c$ is any element of $\mathcal M$ such that $\zero \preceq c$. Since $\Cong$ is a $\preceq$-precongruence, we have that $(\zero,c) \in \Cong$, and hence $c \in \mathcal N_\Cong$. It now follows from the fact that $\mathcal N_\Cong$ is closed under addition, we have $b+c \in \mathcal N_\Cong$ for any $b \in \mathcal N_\Cong$ and $c \in \mathcal M_\textrm{Null}$.
The last assertion is an easy
verification, noting that $\psi_\Cong (a(b_0,b_1)) = ab_0 (-)ab_1 =
a\psi_\Cong (b_0,b_1)$} and \eqref{strict}.
  \end{proof}
%
%

\begin{lem}\label{congsupp22}
With the same notation as above, for any homomorphism $f$,
$$\psi_{\Cong}(\mathrm{preker}  _{\preceq,\operatorname{cong} }f)=
\ker_{\Mod,\mathcal M} f=\psi_{\Cong}(\mathrm{preker}
_{\operatorname{cong} }f) .$$
\end{lem}
\begin{proof}
Suppose that $(b,b') \in \mathrm{preker} _{\preceq,\operatorname{cong} }f$,
i.e., $f(b) \preceq f(b')$. Since $f$ is a homomorphism, we have that
\[
\zero \preceq f(b')(-)f(b) = f(b'(-)b),
\]
showing that $\psi_{\Cong}(b,b')=b'(-)b \in \ker_{\Mod,\mathcal M}
f$. Hence $\psi_{\Cong}(\ker   _{\preceq,\operatorname{cong}
}f)\subseteq \ker_{\Mod,\mathcal M} f$.

Conversely, if $f(b) \succeq \zero$, then we have that $(\zero,b) \in \mathrm{preker}    _{\preceq,\operatorname{cong}
}f$ since $\psi_{\Cong}(\zero,b) = (-)b$. It follows that $\psi_{\Cong}(\ker   _{\preceq,\operatorname{cong}
}f) \supseteq \ker_{\Mod,\mathcal M} f$.

The same argument works for $\mathrm{preker} _{\operatorname{cong}
}f$, since $f(b) \nabla f(b')$   implies $ \zero \preceq
f(b')(-)f(b) = f(b'(-)b).$
\end{proof}
%

\subsubsection{Images} $ $

We have the difficulty that, for any reasonable definition, the
congruence $\preceq$-image of a homomorphism might not contain the
diagonal, so modding it out can destroy well-definedness of
addition.

\begin{defn}\label{definition: wpcg}
    \label{kercokerwithN1}


  \begin{enumerate} \eroman
\item   The  {precongruence image}   of
    a $\preceq$-morphism $
    f:\mathcal M\to \mathcal M'$ is
    \[
    f(\mathcal M)_{\operatorname{pcong}}:=\{(b,b') \in  \mathcal M' \times \mathcal{M}':
    f(c) \succeq b, \ f(c') \succeq b' \text{ for some } c,c' \in
    \mathcal M \}    .
    \]    The \textbf{weak precongruence image} $ f(\mathcal M)_{\operatorname{wpcong} }$ of
  $
    f:\mathcal M\to \mathcal M'$ is  $\Diag_{\mathcal M'} \cup
    f(\mathcal M)_{\operatorname{pcong}
    }.$

The  \textbf{strict precongruence image}   of
    a $\preceq$-morphism $
    f:\mathcal M\to \mathcal M'$ is
    \[
    f(\mathcal M)_{\operatorname{stpcong}}:=\{(b,b') \in  \mathcal M' \times \mathcal{M}':
    f(c) = b, \ f(c') = b' \text{ for some } c,c' \in
    \mathcal M \}\]

\item  Alternatively we can take the \textbf{systemic congruence image}
 $$ f(\mathcal M)_{\operatorname{cong,sys} } =  \Diag_{\mathcal M'} \cup\{ (b,b'): f(c) \succeq
b'(-)b \text{ for some } c\in \mathcal M\}.\footnote{It might be
more natural to use $\preceq $ instead of $\succeq
    $; for example, when $f=0$ we would get $b'(-)b $ a quasi-zero. But then we woulg lose the important
      Proposition \ref{proposition: wpcong is an equivalence}.}
    $$

    %
 \end{enumerate}\end{defn}

\begin{prop}\label{proposition: wpcong is an equivalence}  \begin{enumerate} \eroman
\item
Let $f: \mathcal M \to \mathcal M'$ be a $\preceq$-morphism of
systemic modules. Then the weak precongruence image $f(\mathcal
M)_{\operatorname{wpcong} }$ is an equivalence relation on $\mathcal
M'$.

\item  Let $f: \mathcal M \to \mathcal M'$ be a homomorphism of
systemic modules.  The systemic congruence image $f(\mathcal
M)_{\operatorname{cong,sys}}$ is a congruence  on $\mathcal M'$.

 \end{enumerate}\end{prop}
\begin{proof}
(i) Clearly, $f(\mathcal M)_{\operatorname{wpcong} }$ is reflexive.
Suppose that $(a,b) \in f(\mathcal M)_{\operatorname{pcong}}$. In
other words, there exist $c,c' \in \mathcal M$ such that $f(c)
\succeq a$ and $f(c') \succeq b$. In particular, obviously we have
that $(b,a) \in f(\mathcal M)_{\operatorname{pcong}}$, and hence
$(b,a) \in f(\mathcal M)_{\operatorname{wpcong}}$. This shows that
$f(\mathcal M)_{\operatorname{wpcong} }$ is symmetric. Finally,
suppose that $(a,b),(b,c) \in f(\mathcal M)_{\operatorname{wpcong}
}$. If $(a,b) \in \Diag_{\mathcal M'}$ or $(b,c) \in \Diag_{\mathcal
M'}$, then clearly $(a,c) \in f(\mathcal M)_{\operatorname{wpcong}
}$. Hence, we may assume that $(a,b),(b,c) \in f(\mathcal
M)_{\operatorname{pcong} }$. That is, there exist $d, d', e, e' \in
\mathcal M$ such that
\[
f(d) \succeq a, \quad f(d')\succeq b, \quad f(e) \succeq b, \quad f(e') \succeq c.
\]
This implies that $(a,c) \in f(\mathcal M)_{\operatorname{pcong} }$, showing that $f(\mathcal M)_{\operatorname{wpcong} }$ is transitive.

(ii) The only nontrivial verification is transitivity, i.e., if
 $(b,b'), (b',b'') \in  f(\mathcal M)_{\operatorname{cong,sys} }$ then  $(b,b'') \in  f(\mathcal M)_{\operatorname{cong,sys} }.$
 This is obvious if $b=b'$ or $b' =b'',$ so we can write
 $ f(c) \succeq b(-)b'$
 and  $f(c') \succeq b''(-)b',$ implying $$ f(c+c') =
f(c)+f(c') \succeq b(-)b' + b''(-)b' = b'' (-)b + (b'(-)b')^\circ
\succeq b''(-)b.$$\end{proof}

%
%
%

%
%
\begin{rem}\label{congprob}
 \begin{enumerate}\eroman
\item Unfortunately  $\Diag_{\mathcal M'} + f(\mathcal
M)_{\operatorname{pcong}
 }$ is not a congruence, since it is  reflexive, but need not be symmetric nor transitive. Connes
and Consani \cite[before Lemma~3.2]{CC2} take the smallest
congruence containing this, which is difficult to compute and could
even be everything.

Transitivity is a major obstruction to the theory. For instance,
suppose we are given $f(a)+b =b$ and $b+f(c) = f(c).$ Then we have
$$(b,b)+ (f(a), f(c)) = (b +f(a), b+f(c))= (b,f(c)),$$ which means
we are identifying $b$ with $f(c)$ in any ordered group with $c$
large enough. If transitivity holds and for any $b\in \mathcal M'$
there is $c \in \mathcal M$ for which $f(c) \succeq b$ then the
congruence identifies any two elements of $\mathcal M'$ and
degenerates.

 Alternatively one may be tempted to require that $
\Diag_{\mathcal M'} + f(\mathcal M)_{\operatorname{pcong}
 }$ is already a congruence; i.e., if $b \preceq f(c),$ $b' \preceq
 f(c'),$
 $b'' \preceq  f(c''),$ and
 $(b+d,b'+d), (b'+d',b''+d') \in \Cong,$
with $b'+d = b'+d',$ then $(b+d,b''+d') \in \Cong;$ i.e.,
transitivity requires some extra cancelation or convexity property.

Although this approach is used in the literature, it seriously
restricts the congruences we may use.

\item  In many instances, we might
prefer the weak precongruence image   $ f(C)_{\operatorname{wpcong}
}$. We lose addition, but have a congruence of sets on which $\tT$
acts.

\item The systemic congruence image $ f(\mathcal
M)_{\operatorname{cong,sys} }$ is the best behaved theoretically, in
view of Proposition \ref{proposition: wpcong is an equivalence}(ii),
but it could be prone to blow up to all of $\mathcal M'$ in the
tropical case (when $f(\mathcal M)$ has elements that are
arbitrarily large.)
 \end{enumerate}\end{rem}

\subsubsection{Cokernels}$ $

Here are alternatives  to the categorical definition to be given in
\S\ref{homcat}, which seem more promising for specific results.

 \begin{defn}\label{coke}
Let $f:\mathcal M \to \mathcal M'$ be a $\preceq$-morphism of
systemic modules over a semiring system~$\mathcal A$. Given $b \in
\mathcal M',$ we define $[b]_f = \{ b' \in \mathcal M' : f(c)
\preceq b' (-)b$ for some $c \in \mathcal M\}.$
 The
\textbf{cokernel} of $f: \mathcal M \to \mathcal M'$, denoted
$\coker f$, is the set  $
 \mathcal M' /  f(\mathcal M)_{\operatorname{wpcong} },$ i.e.,
 we mod out by the weak precongruence image.

  The
 \textbf{systemic cokernel} of $f: \mathcal M \to \mathcal M'$, denoted
 $\coker f_{\operatorname{sys}}$, is the set $\{[b]_f : b \in \mathcal
 M'\}.$

\end{defn}

$\coker f$ is not a systemic module since it is not closed under
addition, but is a $\tT$-module.  $\coker f_{\operatorname{sys}}$
provides
   a rather explicit definition of cokernel, although it is viewed in the power
set hypermodule of Example~\ref{reve}, a possible drawback.

we can progress further by
 introducing extra conditions, such as the following.

\begin{defn}\label{reve1} The $f$-\textbf{minimality
condition} is that any element of $[b]_f$ has a unique minimal
element ~$  b_{f}$.

In the presence of the $f$-minimality condition we define the
\textbf{minimal}$f$-\textbf{cokernel}  denoted
$\coker_{\operatorname{min}} f$, as $
 \{b_{f}: b\in \mathcal M'\}$.
 \end{defn}

 The $f$-minimality condition is motivated by the following
 observations.

 \begin{rem}  \begin{enumerate}\eroman \item If $b \succeq f(c) \in f(\mathcal M
 )$ then   $ \zero \in [b]_f$ since we could take $b' =\zero,$
 implying
$  b_{f} = \zero $ since $\zero $ is minimal.
  \item If $b $ is tangible and $b \not \succeq f(c)$ and $\mathcal M$ is $(-)$-bipotent then
  $b' = b.$
    \item From (i) and (ii) we obtain the $f$-minimality
condition for $ \mathcal A^{(n)}$ when $ \mathcal A$ is
$(-)$-bipotent.
    \item $  (b+b')_{f} \preceq  b _{f} + {b'}_{f}$ since at worst one
would descend further. 
\end{enumerate}
\end{rem}

\begin{note*} Each method has its advantages and limitations. $\coker
f$ is the closest to the classical version, although it loses the
 property of being a module, and also becomes rather weak in the absence of negation. $\coker f_{\operatorname{sys}}$ is more in line
with the systemic theory, but it is rather strong.
$\coker_{\operatorname{min}} f$ provides workable results in
tropical situations, but is quite restrictive.
\end{note*}
\subsection{Chain complexes and exact sequences of systemic modules}$ $

Now, we introduce the notion of chain complexes and exact chains of
systemic modules. As before, we will consider various versions
involving surpassing relation $\preceq$ to encapsulate as many as
generalized algebraic structures by one unified framework of
systemic modules.

\begin{defn} \label{t-trivialmorphism}$ $
Let
\begin{equation}
\label {exactseq1} \cdots \longrightarrow  \mathcal K \too {k} \mathcal M  \too
{f}  \mathcal N \to \cdots
\end{equation}
be a sequence of $\mathcal A$-modules with $\preceq$-morphisms. \footnote{Even if we use the term $\succeq$-chain and $\succeq$-exact in this definition, all morphisms are assumed to be $\preceq$-morphisms unless otherwise stated.}
\begin{enumerate} \eroman
\item
\eqref{exactseq1} is a \textbf{$\preceq$-chain at $\mathcal M$} if
$ k({\mathcal K})_\preceq \subseteq \ker_{\Mod,\mathcal M} f.$ A sequence is said to be a $\preceq$-chain complex if it is a $\preceq$-chain at each $\mathcal{A}$-module appearing in a sequence.
\item
\eqref{exactseq1} is a \textbf{$\succeq$-chain at $\mathcal M$} if
 $k({\mathcal K})_\succeq \subseteq \ker_{\Mod,\mathcal M} f.$ A sequence is said to be a $\succeq$-chain complex if it is a $\succeq$-chain at each $\mathcal{A}$-module appearing in a sequence.
\item
\eqref{exactseq1} is a \textbf{chain at $\mathcal M$} if
$ k({\mathcal K}) \subseteq \ker_{\Mod,\mathcal M} f.$ A sequence is said to be a chain complex if it is a chain at each $\mathcal{A}$-module appearing in a sequence.

\item
 A sequence \eqref{exactseq1} is \textbf{$\preceq$-exact}  at $ \mathcal M $
 if $k({\mathcal K})_{\preceq}=\ker_{\Mod,\mathcal M} f;$ in other
 words,   for any $b' \in \ker_{\Mod,\mathcal M}$ there is $b \in \mathcal
 K$ such that $ f(b)\preceq b' $.

\item
 A sequence \eqref{exactseq1} is \textbf{$\succeq$-exact}  at $ \mathcal M $
 if $k({\mathcal K})_{\succeq}=\ker_{\Mod,\mathcal M} f;$  in other
 words,   for any $b' \in \ker_{\Mod,\mathcal M}$ there is $b \in \mathcal
 K$ such that $  f(b)\succeq b'$.

\item   A sequence \eqref{exactseq1} is \textbf{exact}  at $ \mathcal M $
 if $k({\mathcal K})=\ker_{\Mod,\mathcal M} f.$

\item   A sequence \eqref{exactseq1} is \textbf{tangibly exact}  at $ \mathcal M $
 if $k({\tT_\mathcal K})=\tT _{\mathcal M}\cap\ker_{\Mod,\mathcal M} f.$

\item
 A $\preceq$-chain complex \begin{equation} \label{exactseq711} \cdots \to  \mathcal M _n\too {d_n}  \mathcal M _{n-1} \too {d_{n-1}}
\mathcal M _{n-2} \to \cdots\end{equation}  is  \textbf{$\preceq$-exact} if
the sequence $$\mathcal M _{n+1}\too {d_{n+1}}  \mathcal M _{n }
\too {d_{n }}  \mathcal M _{n-1}$$
 is  $\preceq$-exact for each~$n$.

 \item
 A chain complex
 \begin{equation} \label{exactseq712} \cdots \to  \mathcal M _n\too {d_n}  \mathcal M _{n-1} \too {d_{n-1}}
 \mathcal M _{n-2} \to \cdots
 \end{equation}  is  \textbf{exact} if
 the sequence $$\mathcal M _{n+1}\too {d_{n+1}}  \mathcal M _{n }
 \too {d_{n }}  \mathcal M _{n-1}$$
 is exact for each~$n$.

%
 \end{enumerate}\end{defn}

The following is a special case of Definition \ref{t-trivialmorphism} with symmetrization.

\begin{defn}\label{definition: n-exact}
With the same notation as in Definition \ref{t-trivialmorphism}, suppose that the sequences \eqref{exactseq1} and \eqref{exactseq711} are chain complexes.
\begin{enumerate}\eroman
\item
The chain \eqref{exactseq1} is \textbf{N-exact}  at $ \mathcal M $
if the  symmetrized sequence \begin{equation}
\label {exactseq17} \cdots \longrightarrow  \widehat{\mathcal K}
\too {\widehat{k}} \widehat{\mathcal M } \too {\widehat{f}}
\widehat{\mathcal N} \to \cdots
\end{equation}
is exact.
\item
The chain complex \eqref{exactseq711} is
{N}-\textbf{exact} if the sequence $$\widehat{\mathcal M _{n+1}}\too
{\widehat{d_{n+1}}} \widehat{\mathcal M _{n }} \too {\widehat{d_{n
}}} \widehat{\mathcal M _{n-1}}$$
is {N}-exact for each~$n$.
\end{enumerate}
\end{defn}

\begin{remark}
One can easily see from Definition \ref{t-trivialmorphism} that a $\preceq$-chain complex
is a chain complex, and exactness implies $\preceq$-exactness since $k(\mathcal{K}) \subseteq k(\mathcal{K})_\succeq \subseteq \ker_{\Mod,\mathcal M} f$.
\end{remark}

Exactness is more limited in the systemic context.\bigskip

\begin{example}[{Compare with \cite[p.~620]{E}}]\label{reso} The chain $${\mathcal A^\circ} \to {\mathcal A }   \too {\left(\begin{matrix}
 x_1
\\ x_2\\
x_3 \end{matrix}\right)} {\mathcal A }^{(3)} \too
{\left(\begin{matrix}
   x_1 ^\circ & x_3 & (-) x_2\\
(-) x_3 & x_2 ^\circ & x_1
\\ x_2 & (-) x_1 & x_3^\circ\end{matrix}\right)}
  {\mathcal A }^{(3)} \too {\left(x_1
\ x_2\ x_3^\circ \right)}{\mathcal A } $$ is a free chain but
exactness fails in the middle.

Indeed, a vector $(a_1,a_2,a_3)$ is in the kernel, i.e., $a_1
x_1^\circ (-)a_2 x_3 + a_3 x_2 \succeq \zero$ for $a_2x_3 ,a_3x_2
\preceq a_1 x_1;$    $a_2 x_2^\circ (-)a_3 x_1 + a_1 x_3 \succeq \zero$
for $a_3 x_1 = a_1x_3 \succ a_2x_2$ and   $a_3 x_3^\circ (-)a_1 x_2 +
a_2 x_1 \succeq \zero$ for $a_2 x_1 = a_1x_2 \succ a_3x_3$, so we take
$x_1 = a_1$ large and $x_2 = a_2,$ $x_3 = a_3$. This vector
$(a_1,a_2,a_3)$ is not in the image.
\end{example}

%
%
%
%
%
%

Let us see how Definition \ref{definition: n-exact} relates to the literature.

 \begin{rem}  $ $
 \begin{enumerate}\eroman

\item
Any chain complex \begin{equation}\label{exactseq720}   \cdots \to  \mathcal M _n\too {d_n}  \mathcal M _{n-1} \too {d_{n-1}}
 \mathcal M _{n-2} \to \cdots\end{equation} thereby induces a chain
 complex
 \begin{equation} \label{exactseq72} \cdots \to  \widehat{\mathcal M _n}\too {\widehat{d_n}}  \widehat{\mathcal M _{n-1}}
  \too {\widehat{d_{n-1}}}
 \widehat{\mathcal M _{n-2}} \to \cdots,\end{equation}
\end{enumerate}
and thus yields a special case of the double arrow chain complex
introduced by Patchkoria \cite{Pat} and Flores~\cite{Flo}. 
\end{rem}





 \section{(h,$\preceq$)-Projective dimension}\label{projreso}$ $

We recall the various definitions of ``projective'' from \cite{JMR}. See \cite{DeP,IKN,Ka3,KN} for comparison.

\begin{defn}\label{precproj} 

Let $\mathcal P: =   (\mathcal P,\tT_{\mathcal P},(-),\preceq)$ be a systemic module.
\begin{enumerate}\eroman
\item
 $\mathcal P$ is \textbf{projective} if for any  onto
homomorphism of  systemic modules $h: \mathcal M \to \mathcal M'$,
every homomorphism $ f: \mathcal P \to \mathcal M'$ \textbf{lifts}
to a homomorphism $\tilde f: \mathcal P \to\mathcal M$, in the sense
that $h\tilde f =  f.$
\item
$\mathcal P$ is  $\preceq$-\textbf{projective} if for any
$\preceq$-onto $\preceq$-morphism  $h: \mathcal M \to \mathcal M',$
every  $\preceq$-morphism $ f: \mathcal P \to \mathcal M'$
$\preceq$-\textbf{lifts} to a $\preceq$-morphism $\tilde f: \mathcal
P \to \mathcal M$, in the sense that $f \preceq h\tilde f  .$
\item
$\mathcal P$ is  $(\preceq,h)$-\textbf{projective} if for any
$\preceq$-onto homomorphism  $h: \mathcal M \to \mathcal M',$ every
$\preceq$-morphism $ f: \mathcal P \to \mathcal M'$
$\preceq$-\textbf{lifts} to a $\preceq$-morphism $\tilde f: \mathcal
P \to \mathcal M$, in the sense that $f \preceq h\tilde f  .$
\item
$\mathcal P$ is   \textbf{h-projective} if for any  $\preceq$-onto
homomorphism  $h: \mathcal M \to \mathcal M',$ every homomorphism $
f: \mathcal P \to \mathcal M'$ $\preceq$-\textbf{lifts} to a
homomorphism $\tilde f: \mathcal P \to \mathcal M$, in the sense
that $f \preceq h\tilde f  .$
 \item
$\mathcal P$ is   \textbf{N-projective} if the symmetrization
$\widehat{\mathcal P}$ is $h$-projective with regard to the
symmetrized system of Definition \ref{symsyst}.
\end{enumerate}
\end{defn}


\subsection{Resolutions}$ $

\begin{defn}\label{projres}
Let $\mathcal M$ be a systemic module over a semiring system
$\mathcal A$.
\begin{enumerate} \eroman
\item (The classical case)
A \textbf{resolution} of $\mathcal M$ is an exact chain complex
\begin{equation} \label{exactseq77}  \cdots \to \mathcal N _n\too {d_n}  \mathcal N _{n-1} \too {d_{n-1}}
\dots \too {d_{2}}
 \mathcal N _{1}\too {d_{1}}
 \mathcal N _{0}  \too {\varepsilon: = d_0} \mathcal M
\end{equation}
 of
homomorphisms, such that  ${\varepsilon} $ is onto.

\item
A $\preceq$-\textbf{resolution} of $\mathcal M$ is an
$\preceq$-exact chain complex
\begin{equation} \label{exactseq771}  \cdots \to \mathcal N _n\too {d_n}  \mathcal N _{n-1} \too {d_{n-1}}
\dots \too {d_{2}}
 \mathcal N _{1}\too {d_{1}}
 \mathcal N _{0}  \too {\varepsilon: = d_0} \mathcal M
\end{equation}
 of
homomorphisms, such that  ${\varepsilon} $ is onto.

\item (The classical case)
A \textbf{projective resolution} of $\mathcal M$ is a resolution
where each $\mathcal N _i$ is projective.
\item
An \textbf{h-projective resolution} of $\mathcal M$ is a
$\preceq$-resolution where each $\mathcal N _i$ is h-projective.
The null-module kernel $\ker_{\Mod,\mathcal P_n} d _n$ is called the
\textbf{$n$-th
 $\preceq$-syzygy}.

 \item
An \textbf{N-projective resolution} of $\mathcal M$ is a
$\preceq$-resolution  of $\widehat{\mathcal M}$ where each $\mathcal
N _i$ is N-projective.

 \item
A \textbf{free resolution} of $\mathcal M$ is an h-projective
resolution, where each $\mathcal N _i$ is free.
\end{enumerate}\end{defn}

 The definitions of h-projective resolution and N-projective
 resolution match, when $\preceq$ is taken in the context of symmetrization.
 But there are other cases of interest, such as in Definition \ref{reve1}. 

 Projective resolutions lead
naturally to the notion of projective dimension.

\begin{defn}\label{projdim7}
Let $\mathcal M$ be a systemic module over a semiring system
$\mathcal A$.
\begin{enumerate} \eroman
\item  The  \textbf{projective  dimension}
$\operatorname{proj.dim}(\mathcal M)$ of $ \mathcal M$
 is the smallest possible $n$ (if it exists) among all  projective resolutions, such that
 the $n$-th syzygy $\ker_{\Mod,\mathcal P_n} {d_n}$ is   projective. We say $\operatorname{proj.dim}(\mathcal M)=  0$ when $\mathcal M$
 is projective.

\item
The  h-\textbf{projective  dimension}
h-$\operatorname{proj.dim}(\mathcal M)$ of $ \mathcal M$
 is the smallest possible $n$ (if it exists) among all $\preceq$-projective resolutions, such that
 the  $n$-th $\preceq$-syzygy $\ker_{\Mod,\mathcal P_n} {d_n}$ is h-projective.
We say h-$\operatorname{proj.dim}(\mathcal M)=  0$ when $\mathcal M$
is h-projective.

 \item
The  N-\textbf{projective  dimension}
N-$\operatorname{proj.dim}(\mathcal M)$ of $ \mathcal M$
 is the smallest possible $n$ (if it exists) among all N-projective resolutions of $\widehat{\mathcal M}$, such that
 the  $n$-th $\preceq$-syzygy $\ker_{\Mod,\widehat{\mathcal P_n}} {d_n}$ is h-projective.
We say N-$\operatorname{proj.dim}(\mathcal M)=  0$ when $\mathcal M$
is N-projective.
\end{enumerate}\end{defn}

Since a free module over a system is projective and  $h$-projective, and $N$-projective
 (see, \cite[\S
4]{JMR}), one can show that the category of $\tT$-modules has enough
projectives, $h$-projectives, and $N$-projectives; any $\tT$-module $\mathcal M$ has a
projective resolution, simply by taking~$\mathcal P _n$ to be the
free $\tT$-module mapping onto $\ker_{\Mod,\mathcal P_n}{d_{n-1}}$.
This raises the question of which projective $\tT$-modules
are necessarily free.

On the other hand, many chains which are resolutions in the classical sense now fail to be resolutions because of failure of
exactness, as in Example~\ref {reso}. 




\begin{example}
    One can easily see that a module over a system $F[\lambda]$  has an infinite projective dimension.
\end{example}

\subsubsection{Schanuel's Lemma revisited}$ $

The $h$-projective  dimension could be less than the projective
dimension; the projective dimension could even be infinite  while
the $h$-projective dimension 
 were
finite. In classical algebra, the projective  dimension does not
depend on the choice of projective resolution in the following
sense. Let $\mathcal M$ be a module over a commutative ring
$\mathcal A$ of the projective dimension $n$. One can truncate any
projective resolution of $\mathcal M$ in the following sense: If
$\mathcal M$ has a projective resolution
\[
0 \too {\textrm{ }}\mathcal P _m\too {d_m}  \mathcal P _{m-1} \too {d_{m-1}} \dots \too {d_{2}}  \mathcal P_1 \too {d_{1}}
\mathcal P _{0} \too {\varepsilon}  \mathcal M,
\]
with $m \geq n$,
then there exists a projective module $\mathcal P'$ such that the
following is a projective resolution of~$\mathcal M$:

\[
0 \too {\textrm{ }}\mathcal P' \too {\textrm{ }}  \mathcal P _{n-1}
\too {d_{n-1}} \dots \too {d_{2}}  \mathcal P_1 \too {d_{1}}
\mathcal P _{0} \too {\varepsilon}  \mathcal M.
\]
This elementary result is called Schanuel's lemma.

We do not yet have the full analog of Schanuel's lemma  for systems
and $\preceq$-projective dimension, but we still can attain some
reasonable results. For example, \cite[\S~5]{JMR} contains various
versions of Schanuel's Lemma.  Here is a slightly more technical
version, whose essence is in \cite[Theorem 2.8.26]{Row06}.

 Let $\mathcal P   \overset {f }
\longrightarrow \mathcal M$ and $\mathcal P'  \overset {f'}
\longrightarrow \mathcal M'$ be $\preceq$-onto homomorphisms, where
$\mathcal P$ is $h$-projective  and $\mathcal P'$ is $(\preceq,h)$-projective.
Suppose that we have a given $\preceq$-onto homomorphism
$\mu:\mathcal M\to \mathcal M'$. By the definition of
$h$-projective, the $\preceq$-onto homomorphism $\mu f$
lifts to a homomorphism $\tilde{\mu}: \mathcal P \to \mathcal
P'$ such that
\[
\mu f \preceq f'\tilde{\mu}.
\]
We let $\mathcal K = \ker_{\Mod,\mathcal P}f$ and  $\mathcal K' =
\ker_{\Mod,\mathcal P'}f'.$

We refer the readers to \cite[\S 3]{JMR} for the definition and
basic properties of $\preceq$-splitting, a key systemic notion. The
following theorem concerns $\preceq$-splitting of homomorphisms.


\begin{thm}[Semi-Schanuel, homomorphic-version]\label{Sch}
With the same notation as above, there is a $\preceq$-onto
(Definition~\ref{ont}) $\preceq$-splitting homomorphism $g: \mathcal
K' \oplus \mathcal P \to \mathcal P'$, with a $\preceq$-isomorphism
$\Phi: \mathcal K \to
\mathcal K'' $, where $\mathcal K'' = \{ (b',b) \in
\ker_{\Mod,\mathcal K' \oplus \mathcal P}g: b' \preceq \tilde
\mu(b)\}$.
\end{thm}
\begin{proof}

Define the   map
\begin{equation}
g: \mathcal K' \oplus \mathcal P \to \mathcal P', \quad g(b',b) = \tilde \mu(b)(-)b'.
\end{equation}
We first claim that $g$ is a homomorphism which is $\preceq$-onto.
In fact, clearly $g$ is a homomorphism since $\tilde{\mu}$ is a
homomorphism. We first claim that for any $b' \in \mathcal P'$, there
exists $b \in \mathcal P$ such that
\begin{equation}
f'(b') \preceq \mu f(b).
\end{equation}
Indeed, let $c=f'(b')$. Since $\mu$ is a $\preceq$-onto homomorphism, there exists
$x \in \mathcal M$ such that $c \preceq \mu(x)$. Moreover, since $f$
is $\preceq$-onto, we have $b \in \mathcal P$ such that $x \preceq
f(b)$; in particular,
\[
f'(b')=c \preceq \mu(x) \preceq  \mu f (b).
\]
Since $f'$ is a homomorphism,
\[
f'(\tilde \mu(b) (-)b') = f'(\tilde \mu(b)) (-) f'(b')\succeq \mu
f(b) (-) f'(b') \succeq \zero.
\]
Therefore $ \tilde \mu(b) (-)b'\in \mathcal K'.
$
Furthermore,
$
g(\tilde \mu(b)(-)b',b) = \tilde \mu(b)(-)(\tilde \mu(b)(-)b')
\succeq b',
$
implying $g$ is $\preceq$-onto (since $b'$ was taken arbitrarily).
Since $\mathcal P'$ is $(\preceq,h)$-projective, $g$ $\preceq$-splits
(by \cite[Proposition 4.4]{JMR}).

For the last assertion 
 take the map
\[
\Phi:\mathcal K \to \ker_{\Mod,\mathcal K' \oplus
    \mathcal P }g, \quad b \mapsto ( \tilde \mu(b),b).
\]
 $\Phi$ is well-defined since $b \in \mathcal
K$ implies that
\[
f'\tilde \mu (b) \succeq \mu f(b) \succeq \mu  (\zero) = \zero,
\]
showing that $\tilde \mu(b) \in \mathcal K'$. Also, we have
\[
g(\tilde \mu(b),b)=\tilde \mu(b) (-) \tilde \mu(b) \succeq \zero,
\]
showing that $(\tilde \mu(b),b) \in \ker_{\Mod,\mathcal K' \oplus
    \mathcal P }g$. The retract $\Phi'$ of $\Phi$ is given by
    projection onto the second coordinate, i.e.,
    $(b',b)\mapsto b$ when $b \preceq \tilde \mu (b').$ Then $\Phi' \Phi(b)
    =b$ whereas,  when $b '\preceq \tilde \mu (b),$  $\Phi
    \Phi' (b',b) = \Phi(b) =  (\tilde \mu (b),b)
    \succeq (b',b),$ implying $\Phi
    \Phi' \succeq \one_{\mathcal K''}.$

%
%
%
\end{proof}

\section{Homological theory of systems}\label{homth}

Homology in general should be defined using congruences, not
modules, cf.~Connes and Consani \cite{CC,CC2}. We take a similar
approach, but often bypassing congruences in favor of precongruences
in order to broaden the range of applications.

\subsection{Homology of precongruences}$ $

Consider the following sequence of systemic modules.
\begin{equation}\label{eq: sequence of modules}
 \dots \rightarrow \mathcal M_{n+1} \xrightarrow{d_{n+1}} \mathcal M_{n} \xrightarrow{d_{n}} \mathcal M_{n-1} \xrightarrow{d_{n-1}} \cdots
\end{equation}

We assume that \eqref{eq: sequence of modules} is a $\succeq$-chain complex as in Definition \ref{t-trivialmorphism},
that is, $d_n$ is a $\preceq$-morphism and $d_{n+1}(\mathcal{M}_{n+1})_\succeq \subseteq \ker_{\Mod,\mathcal M_n}{d_n}$ for each $n$.

There are four versions of homology, based on the three versions of
image in Remark~\ref{congprob} (Definition \ref{definition: homology
via wpcon}) and a direct generalization from the classical
definition (Definition \ref{definition: strict homology}).   We use
the weak precongruence image in Definition \ref{definition: wpcg} to
define homology, which is going be a set with $\tT$-action.

To be precise, for each $n \in \mathbb{N}$, we have
\[
d_{n+1}(\mathcal M_{n+1})_{\operatorname{wpcong} }:= \Diag_{\mathcal M_n} \cup
d_{n+1}(\mathcal M_{n+1})_{\operatorname{pcong}},
\]
where
\[
d_{n+1}(\mathcal M_{n+1})_{\operatorname{pcong}}=\{(b,b') \in
\widehat{\mathcal{M}}_n: d_{n+1}(c) \succeq b, \ d_{n+1}(c') \succeq
b', \textrm{ for some  }c,c'  \in \mathcal{M}_{n+1}\}.
\]

From Proposition \ref{proposition: wpcong is an equivalence}(i),
$d_{n+1}(\mathcal M_{n+1})_{\operatorname{wpcong} }$ is an
equivalence relation on $\mathcal{M}_n$.

\begin{defn} \label{definition: homology via wpcon}
Letting $C$ be a $\succeq$-chain complex \eqref{eq: sequence of
modules}, we define the following. $Z_n(C)  = \ker_{\Mod,\mathcal
M_n}{d_n}.$
\begin{enumerate}\eroman

\item
$H_n(C)_{\operatorname{sys}}=Z_n(C) /{d_{n+1}(\mathcal
M_{n+1})}_{\operatorname{cong,sys} }$.

\item $H_n(C)=Z_n(C)
/d_{n+1}(\mathcal M_{n+1})_{\operatorname{wpcong}}$. In other words,
$H_n(C)$ is the set of equivalence classes of elements in $Z_n(C)$
under the equivalence relation $d_{n+1}(\mathcal
M_{n+1})_{\operatorname{wpcong}}$.


\item (in line with \cite{CC2})
$H_n(C)_{\operatorname{CC}}=Z_n(C) /{d_{n+1}(\mathcal
M_{n+1})}_\Cong,$ where $\Cong$ is the smallest congruence
containing $\Diag_{\mathcal M'}+{d_{n+1}(\mathcal
M_{n+1})}_{\operatorname{pcong}}$.
\end{enumerate}
\end{defn}

Here is our last version directly generalizing the classical definition.

\begin{definition}\label{definition: strict homology}
Suppose that a sequence of systemic modules \eqref{eq: sequence of
modules} is a $\preceq$-chain complex, and let this be $C$. We define the following. $Z_n(C)  = \ker_{\Mod,\mathcal
        M_n}{d_n}$, and
\[
H_n(C)_{\operatorname{str}}=Z_n(C) /\langle {d_{n+1}(\mathcal M_{n+1})} \rangle,
\]
where $\langle {d_{n+1}(\mathcal M_{n+1})} \rangle$ is the smallest congruence in $\mathcal{M}_{n}$ generated by $d_{n+1}(\mathcal M_{n+1})$.
\end{definition}

If the context is clear, we write $H_n$ for $H_n(C)$, ${H_n}_{,
\operatorname{sys}}$ for ${H_n}(C)_{\operatorname{sys}}$, ${H_n}_{,
\operatorname{str}}$ for ${H_n}(C)_{\operatorname{str}}$, and $Z_n$
for $Z_n(C)$. We say that the homology is \textbf{trivial} if it
consists of a single element.

\begin{lemma}\label{tangind}
With the same notation as above, an action of $\tT$ on $\mathcal{M}_n$ induces an action on $H_n$.
\end{lemma}
\begin{proof}%
The induced action is defined as $a\cdot [b]:=[ab]$ since
$d_{n+1}(ab') = a\cdot d_{n+1}(b')$.
\end{proof}

\begin{prop}\label{exat_is_trivial_Hn_sys}
Suppose that \eqref{eq: sequence of modules} is a $\succeq$-chain complex $C$, where $d_n$ is a homomorphism. The following are equivalent:
    \begin{enumerate}\eroman
        \item
        $C$ is $\succeq$-exact at $\mathcal{M}_n$
        in the sense of Definition \ref{t-trivialmorphism},
        \item ${H_n}_{, \operatorname{sys}}$ is trivial.\end{enumerate}
\end{prop}
\begin{proof}
 $(i) \Rightarrow (ii)$ Suppose that $C$ is $\succeq$-exact at $\mathcal{M}_n$, that is, $Z_n = d_{n+1}(\mathcal{M}_{n+1})_{\succeq}$. Let
$b,b' \in Z_n$. Then there
exist $c, c' \in \mathcal{M}_{n+1}$ such that $d_{n+1}(c)\succeq b$
and $d_{n+1}(c')\succeq b'$ from the $\succeq$-exactness. Since $d_n$ is a homomorphism, we have
\[
d_{n+1}(c(-)c') = d_{n+1}(c)(-)d_{n+1}(c')\succeq b(-)b',
\]
and hence $[b]=[b']$ in ${H_n}_{, \operatorname{sys}}=Z_n(C)
/{d_{n+1}(\mathcal
    M_{n+1})}_{\operatorname{cong,sys} }$, showing that the equivalence class of $b$ is the same as the equivalence class of $b'$ in ${H_n}_{, \operatorname{sys}}$. Therefore, ${H_n}_{, \operatorname{sys}}$
is trivial.

$(ii) \Rightarrow (i)$ Suppose that ${H_n}_{, \operatorname{sys}}$
is trivial, i.e.  for any $b,b' \in Z_n(C)$, we have $[b]=[b']$ in
${H_n}_{, \operatorname{sys}}$. We claim that $Z_n =
d_{n+1}(\mathcal{M}_{n+1})_{\succeq}$. In fact, we only have to
prove that $Z_n \subseteq d_{n+1}(\mathcal{M}_{n+1})_{\succeq}$. Let
$x \in Z_n$. Since ${H_n}_{, \operatorname{sys}}$ is trivial, we
have that $[x]=[\zero]$ in ${H_n}_{, \operatorname{sys}}$. It
follows that there exists $c \in \mathcal M_{n+1}$ such that
\[
d_{n+1}(c) \succeq x(-) \zero  =x,
\]
showing that $x \in d_{n+1}(\mathcal{M}_{n+1})_{\succeq}$.
\end{proof}

\begin{prop}\label{exat_is_trivial_Hn} The following are equivalent for a $\succeq$-chain complex $C$
of \eqref{eq: sequence of modules}:
 \begin{enumerate}\eroman
\item
$C$ is $\succeq$-exact at $\mathcal{M}_n$
 in the sense of Definition \ref{t-trivialmorphism}, \item $H_n$ is trivial.
\end{enumerate}
\end{prop}

\begin{proof} $(i) \Rightarrow (ii)$
 Suppose that $C$ is $\succeq$-exact at
    $\mathcal{M}_n$, that is, $Z_n = d_{n+1}(\mathcal{M}_{n+1})_{\succeq}$. Let
    $b,b' \in d_{n+1}(\mathcal{M}_{n+1})_\succeq$. Then there exist $c, c' \in
    \mathcal{M}_{n+1}$ such that $d_{n+1}(c)\succeq b$ and $d_{n+1}(c') \succeq b'$. Then
    clearly, $(b,b') \in d_{n+1}(\mathcal M_{n+1})_{\operatorname{wpcong}}$, showing that the
    equivalence class of $b$ is same as the equivalence class of $b'$ in $H_n$. Therefore,
    $H_n$ is trivial.

 $(ii) \Rightarrow (i)$ Suppose that $H_n$ is trivial, i.e, $Z_n \times Z_n= d_{n+1}(\mathcal M_{n+1})_{\operatorname{wpcong}}$. From the condition of $\succeq$-chains  we have that $d_{n+1}(\mathcal{M}_{n+1})_\succeq \subseteq Z_n$, so we only have to show the other inclusion. 
    Let $b\in Z_n$, and assume that there exists a pair $(a, a) \in
    \Diag_{\mathcal M_n} \setminus d_{n+1}(\Diag_{\mathcal M_{n+1}})$. Since $d_{n+1}(\mathcal M_{n+1})_{\operatorname{wpcong}}$ is an equivalence relation and $H_n = \{1\}$, the pair $(a, b) \in d_{n+1}(\mathcal M_{n+1})_{\operatorname{wpcong}}$. By assumption $b \neq a$, that is, $(a, b) \not\in \Diag_{\mathcal M_n}$. This means that the pair $(a, b)$ must be an element of $d_{n+1}(\mathcal{M}_{n+1})_{\operatorname{pcong}}$, more precisely, there exists $c$ such that $a \preceq d_{n+1}(c)$. Hence $a \in d_{n+1}(\mathcal{M}_{n+1})_\succeq$, for every $(a,a)$ in  $\Diag_{\mathcal M_n}$. This implies that $\Diag_{\mathcal M_n} \subseteq d_{n+1}(\mathcal M_{n+1})_{\operatorname{pcong}}$ and $d_{n+1}(\mathcal M_{n+1})_{\operatorname{wpcong}} = d_{n+1}(\mathcal M_{n+1})_{\operatorname{pcong}}$.
\end{proof}

Note that in order for Definition~\ref{definition: homology via
wpcon}(i) to make sense we require ${d_{n+1}(\mathcal
M_{n+1})}_{\operatorname{cong,sys} }$ to be a congruence, which
happens when the maps $d_n$ are homomorphisms, as shown in
Proposition~\ref{proposition: wpcong is an equivalence}.

For $H_n(C)_{\operatorname{CC}},$ the analog of Lemma~\ref{tangind}
goes through with the same proof. A statement similar to
Proposition~\ref{exat_is_trivial_Hn_sys} holds in this new setting
as well.

\subsubsection{Specific examples}$ $

Here are some examples of systemic modules defined in terms of chains, and their homologies.

\begin{example}\label{homex}
As one can see in Definitions \ref{definition: homology via wpcon}
and \ref{definition: strict homology}, there are several possible
homologies which one can define. This is due to the generality of
systems. Here are some examples paralleling classical examples.

\begin{enumerate}\eroman
\item
Let $\ell_a$ be the  left multiplication map from $\mathcal M$ to $\mathcal M$. Consider the following $\preceq$-chain complex:
\[
K(a): \zero\hookrightarrow \mathcal M  \too {\ell_a}
\mathcal M
\]
In this case, one obtains
\[
H_1(K(a))_{\operatorname{str}}=\Ann _{\mathcal M}a /\langle {\zero} \rangle = \Ann _{\mathcal M}a.
\]
\item
Generalizing the above, given  $a,a' \in \tT$ we have the commuting diagram
$$\begin{matrix}
 \zero\to
&\mathcal M\too {\ell_a}  &\mathcal M
\\  & \downarrow {\ell_{a'}} & \downarrow {\ell_{a'}}\\
 \zero\to
&\mathcal M\too {\ell_a}  &\mathcal M
\end{matrix},$$
which can be rewritten as the chain complex $$\label{2chain}
K(a,a'): \zero\to \mathcal M  \too {\binom{a'}a} \widehat {\mathcal
M} \too {(-a,a')} \mathcal M.$$  If $a$ is
 null-regular then $H_1(K(a,a'))_{\textrm{str}} = [a:a']/a \mathcal M,$ defined
as $\{ b \in \mathcal M: a'b \succeq a \mathcal M\}$,
  in analogy
 to \cite[p. 425]{E}.


  \item (suggested by Lorscheid) The truncated systemic module $\mathcal A [\la]/\Cong,$ where $$\Cong
  = \{ (f+ \la^{m+1} g, f+\la^{m+1} h): \deg f = m\}.$$ This mods out
  all polynomials of degree $>m$ since we can take $f=g=0$ or
  $f=h=0$. Then $H_{{m+1},str}$ is trivial.

   \item The Grassmann (exterior) semialgebra system of a module $\mathcal M$ (over a commutative semiring) is
    treated (and exploited) in  \cite{GaR} where we formally define $(-)$ on the elements of the tensor
       algebra of degree $\ge 2$ by $(-)(v \otimes v') = v' \otimes
       v$. In particular all elements $v \otimes v$ are
       quasi-zeros (but nonzero), but we can further mod out $v\otimes v$ to $\zero.$ In case $\mathcal M$ is free with base $\{ b_i: i \in I\}$ where $I$ is an ordered set,
       the tangible elements now are spanned by $\{ (\pm) b_{i_1}\dots
       b_{i_m}: i_1 < \dots < i_n \}$.
       This enables us to define, for any $b\in \mathcal M$, the \textbf{Koszul complex}
$$K(b): \mathcal M \to \wedge^2 \mathcal M \to \dots \to \wedge^n \mathcal M \too {d_n} \dots$$
where $d_n (v) = v \wedge b.$ The homology of the Koszul complex
casts light on regular sequences (\S\ref{regseq0}).
  \end{enumerate}
\end{example}

Many classical examples, say from D.~Rogalski \cite{Rog}, can be put
in this systemic context.

 \subsubsection{Regular sequences and homology}\label{regseq0}$ $

Homology ties in with regular sequences, which classically lead in
to Koszul complexes, cf.  \cite[\S17]{E}.

\begin{defn} \label{regseq}
A    \textbf{regular sequence} on a $\preceq$-module $\mathcal M$ is
a subset $S = \{ b_1, \dots, b_t \} \subseteq \mathcal A$ for which $S
\mathcal M \ne \mathcal M$ but $b_j$ is null-regular  for  $S_j
\mathcal M \subseteq \mathcal M$   for all $0 \le j \le t-1$, where
$S_j = \{ b_1, \dots, b_j \} $.
\end{defn}

\begin{lem} (As in \cite[p. 425]{E}) \begin{enumerate}\label{regseq1}\eroman
   \item $\{ \la_1, \dots, \la_t\}$ is a regular sequence in the
   polynomial system $\mathcal A [\la_1, \dots, \la_t]$.
    \item For $t=1$, the sequence $S = \{ b \}$ is regular if and only if the
    $H_{1}(K(b))_{str}$ of the complex $$\zero \to \mathcal A
    \too{b}  \mathcal A$$ is trivial.
     \item For $t=2$, and $a_1$ null-regular,  the sequence $S = \{ a_1, a_2 \}$ is regular
     if and only if the
    the systemic homology of the complex  $$\label{2chain}
K(a_1,a_2): \zero\to \mathcal M  \too {\binom{a_2}a_1} \widehat
{\mathcal M} \too {(-a_1,a_2)} \mathcal M$$ of
Example~\ref{homex}(ii) is trivial.
   \end{enumerate}
\end{lem}
\begin{proof} (i) The regularity condition is clear  for $t=1.$  In
general
   modding out $ \{ \la_1, \dots, \la_j \}$ still leaves the
   polynomial system $\mathcal A [\la_{j+1}, \dots, \la_t]$, which satisfies
   the condition by induction.

   (ii) By Proposition~\ref{exat_is_trivial_Hn} and
   Example~\ref{homex}, since each assertion says $\Ann _{\mathcal M}a= \zero.$

(iii) We follow the fourth paragraph of \cite[p. 425]{E}. An element
is in the $\preceq$-image of \eqref{2chain} iff it surpasses $(ca',
ca),$ so $[a:a']$ corresponding to the elements of the image are the
elements surpassing $\mathcal A a.$ The assertion follows when $a_1$
is null-regular.
\end{proof}

\subsubsection{The differential approach}$ $

We can define a specific chain complex from a given  module
$\mathcal M$ with negation, applying symmetrization to the classical
theory given in \cite{C,Kar,Sha}. We define  $(-)^1 b = (-)b$ and
inductively $(-)^k b = (-)(-)^{k-1}b.$

Tensor products over semirings parallel tensor products over rings,
and are well studied in the literature \cite{Ka1,Ka2,Ka3}; the
systemic version is given in \cite[\S~6.4]{Row16}, where it is
stipulated that $(-)(b \otimes b') = ((-)b) \otimes b'$ for all
$b,b'$. Given a commutative semiring system $(\mathcal A, \tT, (-),
\preceq)$ and a systemic module $\mathcal M$ we define $\mathcal
M^{\otimes n} = \mathcal M \otimes \mathcal M^{\otimes {n-1}}.$

 One needs to use homomorphisms for the theory to contain tensor products
in the context of monoidal categories, as illustrated in
\cite[Example 7.7]{JuR1} and even more extreme examples of Gaubert
\cite{Ga}.

\begin{example} Here is the systemic version of \cite{C}, \cite{Sha}. Given a semiring  $\mathcal M$ which
is a  module with negation over a commutative semiring system
$(\mathcal A, \tT, (-), \preceq)$, define ${\mathcal
M}_{\operatorname{ab}}= \{ b b' (-) b'b : b,b' \in \mathcal M\}.$
(${\mathcal M}_{\operatorname{ab}}$ is null when $\mathcal M$ is
commutative.) Define the differential map $d_n: \mathcal M^{\otimes
n} \to \mathcal M^{\otimes n-1}$ by
$$d_n (b_1 \otimes \dots  \otimes b_n ) = b_n b_1 \otimes \dots  \otimes b_{n-1}  + \sum_{k=1}^{n-1} (-)^k b_1 \otimes \dots \otimes b_{k-1}\otimes b_{k}b_{k+1}\otimes b_{k+2}\otimes \dots \otimes b_{n}.$$
Then each term in $d_k d_{k-1}$ appears with its quasi-negative,
implying $d_k d_{k-1}$ is null for each $k$, and we have a chain
$\mathcal M^{\otimes n} \too{d_n} \mathcal M^{\otimes n-1}
\too{d_{n-1}}\dots \too{d_{1}} {\mathcal M}_{\operatorname{ab}}.$
(since $d_{1}(b_1 \otimes b_2) = b_2b_1 (-) b_1 b_2$).

By modding out the congruence spanned by all $(b_1 \otimes \dots
\otimes b_n,   b_n b_1 \otimes \dots \otimes b_{n-1}),$  one obtains
the systemic version of cyclic homology.
 Instead of homomorphisms of tensor products, one
could consider $n$-linear maps, and work in the semigroup algebra of
the cartesian power $\mathcal M^{(n)}$.

The  version $d_n: \widehat{\mathcal M}^{\otimes n} \to
 \widehat{\mathcal M}^{\otimes n-1}$ under symmetrization for an arbitrary semi-algebra over a
semiring is to symmetrize the differential map, i.e., put the
positive parts in the first component and the negative parts in the
second component,  taking $$d_n  \big((b_1,b_1') \otimes \dots
\otimes (b_n,b'_n)\big) = (c,c')$$ where

  \begin{equation}\label{fix1}\begin{aligned}   c= b_n b_1+b_n'b_1'  & \otimes (b_1,b_1') \otimes \dots  \otimes (b_{n-1},b'_{n-1}) \\&  + \sum_{k \text{ even} } (b_1,b_1') \otimes \dots  \otimes (b_{k-1},b'_{k-1})\otimes (b_{k}b_{k+1}+ b'_{k}b'_{k+1})\otimes (b_{k+2},b'_{k+2}) \otimes \dots  \otimes (b_n,b'_n) \\ &
 + \sum_{k \text{odd} } (b_1,b_1') \otimes \dots  \otimes (b_{k-1},b'_{k-1})\otimes b'_{k}b'_{k+1}\otimes  (b_{k+2},b'_{k+2}) \otimes \dots  \otimes (b_n,b'_n),\end{aligned}\end{equation}
\begin{equation}\label{fix2}\begin{aligned} c'=  b_n b'_1+b_n'b_1 &
\otimes (b_1,b_1') \otimes \dots  \otimes (b_{n-1},b'_{n-1}) \\ & +
\sum_{k \text{ odd} } (b_1,b_1') \otimes \dots  \otimes
(b_{k-1},b'_{k-1})\otimes (b_{k}b_{k+1}+ b'_{k}b'_{k+1})\otimes
(b_{k+2},b'_{k+2}) \otimes \dots  \otimes (b_n,b'_n) \\ &
 + \sum_{k \text{even} } (b_1,b_1') \otimes \dots  \otimes (b_{k-1},b'_{k-1})\otimes b'_{k}b'_{k+1}\otimes  (b_{k+2},b'_{k+2})
  \otimes \dots  \otimes (b_n,b'_n) .\end{aligned}\end{equation}\end{example}



\subsection{Chains of systemic modules}\label{homth1}

\begin{lem}\label{inducedmap}
Consider the following commutative diagram of systemic modules over
a semiring system~$\mathcal A$:
\begin{equation}\label{commutativediagram}
\begin{tikzcd}
\mathcal M' \arrow{r}{q}\arrow{d}{f} & \mathcal N'\arrow{d}{g}
\\
\mathcal M\arrow{r}{l} & \mathcal N
\end{tikzcd}
\end{equation}
with $\preceq$-morphisms  $f$, $g$, $q$, and  $l$. The commutative
diagram \eqref{commutativediagram} induces the following natural maps:
\begin{enumerate}\eroman
  \item
A $\preceq$-morphism $\ker_{\Mod,\mathcal M'} f \to \ker_{\Mod,\mathcal N'} g$ given by $x_0 \mapsto  q(x_0)$.
  \item
A $\tT_\mathcal{A}$-equivariant map $[l]: \coker f \to \coker g$ given by $[y] \mapsto [l(y)]$, where $[y]$ is the equivalence class of $y \in \mathcal{M}$ in $\coker f$.
  \item
A precongruence $\preceq$-morphism $\pker _{\operatorname{cong} }f\to \pker _{\operatorname{cong}
  }g$ given by $(x_0,x_1) \mapsto  (q(x_0),q(x_1)).$
 \item
A precongruence $\preceq$-morphism $\pker
_{\preceq,\operatorname{cong} }f\to \pker
_{\preceq,\operatorname{cong}
  }g$ given by $(x_0,x_1) \mapsto  (q(x_0),q(x_1)).$
\end{enumerate}
\end{lem}
\begin{proof}
(i) Suppose $x_0 \in \ker_{\Mod,\mathcal M'} f.$ By definition,
$f(x_0)\succeq \zero,$ and hence
\[
gq(x_0)=l f(x_0)\succeq  l (\zero) \succeq \zero,
\]
proving $q(x_0) \in \ker_{\Mod,\mathcal N'} g.$ This map is clearly a $\preceq$-morphism.

(ii) We claim that $[l]$ is well-defined. In fact, if $[y]=[y'] \in
\coker f$, then there exist $c,c' \in \mathcal{M}'$ such that $f(c) \succeq y$ and $f(c') \succeq y'$. It follows that
\[
g(q(c))=l(f(c)) \succeq l(y) \textrm{ and } g(q(c'))=l(f(c')) \succeq l(y'),
\]
showing that $[l(y)]=[l(y')]$. This map is clearly $\tT$-equivariant since for any $t \in \tT$, we have
\[
[l](t\cdot[y]) = [l]([ty])=[l(ty)]=[tl(y)]=t\cdot [l(y)].
\]


(iii)  Let $(x_0,x_1) \in \pker _{\operatorname{cong} }f$, that is
$\zero \preceq f(x_0)(-)f(x_1)$. It follows that
\[
gq(x_0) (-) gq(x_1) \succeq gq(x_0(-)x_1) = lf(x_0(-)x_1) \succeq l(\zero) \succeq \zero,
\]
showing that $(q(x_0),q(x_1)) \in \pker _{\operatorname{cong}}g$. This map is clearly a precongruence $\preceq$-morphism.

(iv) Let $(x_0,x_1) \in \pker
_{\preceq,\operatorname{cong} }f$, that is $f(x_0) \preceq f(x_1)$. From the commutative diagram \eqref{commutativediagram}, we have that
\[
g(q(x_0))=l(f(x_0)) \preceq l(f(x_1))=g(q(x_1)),
\]
and hence $(q(x_0),q(x_1)) \in \pker
_{\preceq,\operatorname{cong}
}g$. In particular, the map $\pker
_{\preceq,\operatorname{cong} }f\to \pker
_{\preceq,\operatorname{cong}
}g$ sending $(x_0,x_1)$ to $(q(x_0),q(x_1))$ is well defined. Finally, it is clear that this map is a precongruence $\preceq$-morphism.
\end{proof}
%
%
%
%
%
%
%

Consider the following commutative diagram of $\preceq$-morphisms:
\begin{equation}\label{snakeseq881}
\begin{tikzcd}
\mathcal M' \arrow{r}{q}\arrow{d}{f} & \mathcal N'\arrow{d}{g}\arrow{r}{p} & \mathcal{L}'\arrow{d}{h}
\\
\mathcal M\arrow{r}{l} & \mathcal N \arrow{r}{r} & \mathcal{L}
\end{tikzcd}
\end{equation}
 where $p$ is a
$\succeq$-onto homomorphism, $g$ is a homomorphism, and $l$ is a
$\preceq $-monic homomorphism.

We want a usable result that says  there exists a natural
$\tT$-equivariant map $d$ sending $ \ker_{\Mod,\mathcal L'}h$ to
some cokernel. Unfortunately well-definedness does not hold in
general, due to the difficulties with transitivity noted in
Remark~\ref{congprob}. We see two ways to overcome this, given the
set-up of \eqref{snakeseq881}.

\begin{thm}[Systemic connecting map]\label{snake1}

In addition to the hypotheses above, suppose that  the  top row is
$\preceq$-exact and the bottom row is $\succeq$-exact
(cf.~Definition~\ref{t-trivialmorphism}). Then there exists a
natural $\tT_\mathcal{A}$-equivariant map $$d: \ker_{\Mod,\mathcal
L'} h \to \coker f_{\operatorname{sys}}$$ (given in the proof). If
the $f(\mathcal M')$-minimality condition of Definition \ref{reve1}
holds, there exists a natural $\preceq$-morphism $d:
\ker_{\Mod,\mathcal L'} h \to \mathcal M$.
\end{thm}


\begin{proof}
The construction of $d$ is similar to the classical case.
 For each
$x_0 \in\ker_{\Mod,\mathcal L'} h$, we want to define an element
$d(x_0)$ in $\coker f_{\operatorname{sys}}$.
 Since $p$ is $\succeq$-onto, we have
$v_0\in {\mathcal N'}$ such that ${p}(v_0) \succeq x_0$. Hence,
\begin{equation}
rg(v_0) = hp(v_0)\succeq h(x_0) \in \mathcal
L_{\operatorname{Null}},
\end{equation}
so $g(v_0) \in \ker_{\Mod,\mathcal N} r$. From the
$\succeq$-exactness of the bottom row
 we have $u_0 \in
\mathcal{M}$ such that
\[
l(u_0) \succeq g(v_0).
\]
We then define $d(x_0)$ to be the image of $u_0$ in $\coker
f_{\operatorname{sys}}$.

We claim that $d$ is well-defined. In fact, suppose that we  have
$p(v_0') \succeq x_0$  as well as $ p(v_0) \succeq x_0$. Then
\begin{equation}
p(v_0'(-)v_0) = p(v_0')(-) p(v_0) \succeq x_0  (-)x_0 \succeq \zero
\end{equation}
since $p$ is a homomorphism. In particular, $v_0'(-)v_0 \in
\ker_{\Mod,\mathcal N'} p$. Now, the $\preceq$-exactness of the top
row shows that there exists $ y \in \mathcal M'$ such that
\[
{q}(y) \preceq v_0'(-)v_0.
\]
Since $g$ and $l$ are homomorphisms, 
\begin{equation}
{l}{f}(y ) = {g}{q}(y ) \preceq {g}(v_0'(-)v_0) = g(v_0') (-) g(v_0)
\preceq l(u_0')(-)l(u_0) = l(u_0'(-)u_0).
\end{equation}
Now, since ${l}$ is $\preceq$-monic we get  $ {f}(y)\preceq u_0'(-)
u_0$, so $u_0'$ and $ u_0$ map to the same element in $\coker
f_{\operatorname{sys}}$.

Clearly this map $d$ is $\tT_{\mathcal{A}}$-equivariant since for
each relation in the proof, multiplying an element of
$\tT_{\mathcal{A}}$ does not change anything.

If the $f(\mathcal M')$-minimality condition holds, take such $ u_0$
$f$-minimal (using the minimality condition). In proving uniqueness,
the minimality condition implies $u_0'=u_0$. Furthermore, if
$d(x_0+x_1)=u_2$, the minimality condition implies that $u_2 \preceq
u_0+u_1$, showing that $d$ is a $\preceq$-morphism in this case.
\end{proof}

%

 Here is an alternate version that holds in tropical math. In
the case that $(\mathcal A, \tT, (-))$ is  metatangible, we have a
natural algebraic semilattice structure on $\tT$ by putting $a
\wedge a = a \vee a = a,$ and $a \vee a' = a + a'$ for $a \ne a'.$
We say that a semilattice   is \textbf{complete} if any set has a
sup. A systemic module $\mathcal{M}$ is \textbf{complete} if its
underlying tangible set is complete. Thus, when $\mathcal{M}$ is
$(-)$-bipotent we can take infinite sums.

 Write $\tT_{ker,\mathcal L'}$ for the tangible elements of
$ \ker_{\Mod,\mathcal L'} h $, and analogously for
$\tT_{ker,\mathcal M'}$ and $\tT_{ker,\mathcal N'}$.
\begin{thm}(Tropical Connecting map)\label{snake01} Suppose that $\mathcal{M}$ is a free  module
 over a complete metatangible triple $(\mathcal A, \tT, (-))$.
 Then,  given the
set-up of \eqref{snakeseq881}, with $g$ tangible
(Definition~\ref{ordmp}) and the bottom row tangibly exact, there
exists a natural $\tT_\mathcal{A}$-equivariant map $d$ from
$\tT_{ker,\mathcal L'}$ to $\tT _M$ given in the proof.
\end{thm}

\begin{proof}  The construction of $d$ is similar to the previous case. We may assume that $\mathcal{M}=\mathcal{A}$ since $\mathcal{M}$ is free. For each
$x_0 \in \tT_{ker,\mathcal L'}$, we want to define an element
$d(x_0)$ in $\coker f$. Since $p$ is $\succeq $-onto, we have $v\in
{\mathcal N'}$ such that $x_0 \preceq {p}(v)$, and $v$ must be
tangible. (This follows from \cite[Theorem~4.31]{Row16}.) Take $v_0$
to be the sup of all such $v$. Hence,
\begin{equation}
rg(v_0) = hp(v_0)\succeq h(x_0) \in \mathcal
L_{\operatorname{Null}},
\end{equation}
so $g(v_0) \in \ker_{\Mod,\mathcal N} r$. It follows from the
tangibly exactness
of the bottom row that 
\begin{equation}\label{defined}
\exists u_0 \in  \tT_{\mathcal M}  \textrm{ such that } g(v_0)
\preceq {l}(u_0) .
\end{equation}
We   define $d(x_0)$ to be the image of $ u_0  $ in $\coker {f}$;
$d$ is well-defined since there were no choices involved in its
definition. (If we had instead $u_0'$ with $ g(v_0) \preceq
{l}(u_0') $ then  $\zero \preceq g(v_0)(-) g(v_0)\preceq
{l}(u_0')(-){l}(u_0),  $ implying $\zero \preceq u_0' (-) u_0, $ so
$u_0' = u_0$ since both are tangible.


Finally, one can easily see that this map is indeed
$\tT_\mathcal{A}$-equivariant.

\end{proof}

\begin{rem} When $\mathcal{A}$ is  $(-)$-bipotent of height 2, $d$
provides a map $ \ker_{\Mod,\mathcal L'} h \to \tT_M$. Thus
Theorem~\ref{snake01} has useful content.
\end{rem}

%

\begin{lem}\label{chai}
Suppose that we have the same respective set-ups as in Theorems
\ref{snake1} and \ref{snake01}, with  all maps tangible in the
latter situation. We have the following respective sequences of
$\tT_\mathcal{A}$-modules with $\tT_{\mathcal{A}}$-equivariant maps:
\begin{itemize}
\item[(a)]
\begin{equation}\label{eq:snake sequence}
\ker_{\Mod,\mathcal M'}{f} \too {\tilde q} \ker_{\Mod,\mathcal
    N'}{g} \too {\tilde p} \ker_{\Mod,\mathcal L'} h \overset{d}
\longrightarrow \coker({f})_{\operatorname{sys}} \too {\bar
l}\coker({g})_{\operatorname{sys}}\too {\bar
    r}\coker({h})_{\operatorname{sys}},
\end{equation}

\item[(b)]
\begin{equation}\label{eq: snake sequence1}
\tT_{ker,\mathcal M'}{f} \too {\tilde q}\tT_{ker,\mathcal N'}{g}
\too {\tilde p}\tT_{ker,\mathcal L'} h \overset{d} \longrightarrow
\tT_{\mathcal M}\too {\bar l}\tT_{\mathcal N}\too {\bar
    r}\tT_{\mathcal L},\end{equation}
\end{itemize}

where:

\begin{enumerate}
\item $\tilde{q}$ and $\tilde{p}$ are the respective restrictions of
$q$ and $p$,

\item $d$ is the connecting map in Theorem \ref{snake1}, \ref{snake01} respectively,   and

\item  $\bar{l}$ and $\bar{r}$ are the respective restrictions of $l$ and $r$.
\end{enumerate}
\end{lem}
\begin{proof}
$\textbf{(a)}:$ We have to prove that $\tilde{q},\tilde{p},\bar{l},\bar{r}$ are well defined. First, we claim that $\tilde{q}$ is well-defined. In fact, take any $x \in \ker_{\Mod,\mathcal M'}{f}$. From the commutative diagram \eqref{snakeseq881}, we have that
\[
g(q(x))=l(f(x)) \succeq \zero,
\]
showing that $q(x) \succeq \zero$, or equivalently, $\tilde{q}(x)
\in \ker_{\Mod,\mathcal
    N'}{g}$, showing that $\tilde{q}$ is well defined. One can easily see that $\tilde{p}$ is well-defined by the same argument.
Next, we claim that $\bar{l}$ and $\bar{r}$ are well-defined.
In fact, if $[b]=[b']$ in $\coker({f})_{\operatorname{sys}}$, then we have $c \in \mathcal M'$ such that $f(c) \preceq b'(-)b $. From the commutative diagram \eqref{snakeseq881}, we have that
\[
g(q(c))=l(f(c)) \preceq l(b') (-) l(b),
\]
showing that $[l(b')]=[l(b)]$ in $\coker(g)$. Similarly, $\bar{r}$ is well-defined.

$\textbf{(b)}:$ One can easily see that $\tilde{q}$, $\tilde{p}$,
$\bar{l}$ and $\bar{r}$ are well-defined from the exact same
argument as in $(a)$ with the fact that the maps send tangible
elements to tangible elements.
\end{proof}

 Then, we have the following
version of the Snake Lemma. The hypothesis in (i) is needed since
otherwise its proof would  reverse the direction.

\begin{thm}(Weak Systemic Snake Lemma)\label{theorem: snake lemma}
Suppose that we have the same respective set-up as in
Theorem~\ref{snake1}. Notation as above, the sequence
\eqref{eq:snake sequence} satisfies the following:
\begin{enumerate} \eroman
\item
If the top row of \eqref{snakeseq881} is exact, then
\[
\tilde{q}(\ker_{\Mod,\mathcal M'}{f})_\succeq = \ker_{\Mod,\mathcal H}{\tilde{p}},
\]
where $\mathcal {H}=\ker_{\Mod,\mathcal
    N'}{g}$.
\item
\[
\tilde{p}(\ker_{\Mod,\mathcal N'}{g}) \subseteq \{b \in \ker_{\Mod,\mathcal
    L'} h : d(b)=[\zero]\},
\]
where $[\zero]$ is the equivalence class of $\zero$ in $\coker(f)_{\operatorname{sys}}$.

\item
\[
d(\ker_{\Mod,\mathcal L'} h) \subseteq \{[b] \in \coker(f)_{\operatorname{sys}} :
\bar{l}([b])=[\zero]\},
\] where $[\zero]$ is the equivalence class
of $\zero$ in $\coker(g)$.

\item
If the bottom row of \eqref{snakeseq881} is exact, then
\[
\bar{l}(\coker(f)_{\operatorname{sys}}) \subseteq \{b' \in
\coker(g)_{\operatorname{sys}} : \bar{r}(b')=[\zero]\},
\]
where $[\zero]$ is the equivalence class of $\zero$ in $\coker(h)_{\operatorname{sys}}$.
\end{enumerate}
\end{thm}
\begin{proof}
\textbf{(i)} For $b \in \tilde{q}(\ker_{\Mod,\mathcal
M'}{f})_\succeq$, clearly $p(b) \succeq \zero$ from the exactness of
the top row of \eqref{snakeseq881}. Furthermore $g(b) \succeq
\zero$. Indeed, from our assumption, we have  $u \in
\ker_{\Mod,\mathcal M'}{f}$ such that $q(u)= b$. It follows from the
commutative diagram \eqref{snakeseq881} that
\[
g(b)= g(q(u))=l(f(u)) \succeq l(\zero) \succeq \zero,
\]
showing that $g(b) \succeq \zero$. This shows that  $\tilde{q}(\ker_{\Mod,\mathcal M'}{f})_\succeq \subseteq \ker_{\Mod,\mathcal H}{\tilde{p}}$. Conversely, let $a \in \ker_{\Mod,\mathcal H}{\tilde{p}}$, that is, $p(a)\succeq \zero $ and $g(a) \succeq \zero$. But, since $q(\mathcal M')= \ker_{\Mod,\mathcal N'}{p}$ from the exactness assumption, there exists $x \in \mathcal M'$ such that $q(x)=a$. Furthermore we have
\[
l(f(x))=g(q(x))=g(a) \succeq \zero,
\]
showing that $f(x) \succeq \zero$ since $l$ is $\preceq$-monic. In particular, $x \in \ker_{\Mod,\mathcal M'}{f}$. This implies that $a \in \tilde{q}(\ker_{\Mod,\mathcal M'}{f})_\succeq$.
\vspace{0.3cm}

\textbf{(ii)} Let $x \in \ker_{\Mod,\mathcal N'}{g}$. We know from
Lemma \ref{chai} that $b:= p(x) \in \ker_{\Mod,\mathcal L'} h$. We
claim that $d(p(x))=[\zero]$. In fact, from the definition of $d$ in
Theorem \ref{snake1}, we have the following:
\begin{equation}\label{eq: Theorem 5.12 (ii)}
\begin{tikzcd}
 & x\arrow{d}{g}\arrow{r}{p} & b
\\
u\arrow{r}{l} & g(x)  &
\end{tikzcd} \quad \textrm{ such that } p(x)=b,~~ l(u)\succeq g(x) \succeq \zero,~~d(b)=[u].
\end{equation}
Since $l$ is $\preceq$-monic, \eqref{eq: Theorem 5.12 (ii)} implies that $u \succeq \zero$, showing that $[u]=[\zero]$ in $\coker(f)_{\operatorname{sys}}$ since, in this case, we have that $f(\zero)=\zero \preceq u (-) \zero$. This shows that $\tilde{p}(\ker_{\Mod,\mathcal N'}{g}) \subseteq \{b \in \ker_{\Mod,\mathcal
    L'} h : d(b)=[\zero]\}$.
\vspace{0.3cm}


\textbf{(iii)} Suppose $x_0 \in \ker_{\Mod,\mathcal L'} h$. We can
find $u_0$ and $v_0$ as in the proof of Theorem~\ref{snake1} such
that $d(x_0)=[u_0]$, $p(v_0) \succeq x_0$, $r(g(v_0)) \succeq
\zero$, and $l(u_0) \succeq g(v_0)$. But, the last condition that
$l(u_0) \succeq g(v_0)$ directly implies
\[
\bar{l}(u_0)=[l(u_0)]=[\zero]
\]
in $\coker(g)_{\operatorname{sys}}$, showing that
\[
d(\ker_{\Mod,\mathcal L'} h) \subseteq \{[b]\in \coker(f)_{\operatorname{sys}}:
\bar{l}([b])=[\zero]\}.
\]




\vspace{0.3cm}

\textbf{(iv)} Suppose that $[b] \in \coker(f)_{\operatorname{sys}}
$. Then we have, from the exactness of the bottom row, $r(l(b)) \succeq \zero$. In particular, we have
\[
\bar{r}(\bar{l}(b))=\bar{r}([l(b)])=[r(l(b))]=[\zero],
\]
showing that $\bar{l}(\coker(f)_{\operatorname{sys}}) \subseteq \{b' \in
\coker(g)_{\operatorname{sys}} : \bar{r}(b')=[\zero]\}$.
\end{proof}

\begin{remark}(Minimal Weak Snake Lemma)
Assume the $f$-minimality condition holds, with the set-up of \ref{snake01}, where all maps are tangible. With the same set-up as in Theorem \ref{theorem: snake lemma}, one can easily prove a similar result (``Minimal Weak Snake Lemma'') about the sequence \eqref{eq: snake sequence1}, where $\tilde{\phantom{w}}$ denotes the restriction to
the kernel and $\bar{\phantom{w}}$ denotes the induced  map on the image.
\end{remark}

\section{Categorical aspects in the perspective of homological categories}\label{Grand}$ $

The objective of this section is to show how the spirit of this
paper, via triples, is consistent with work in the  more abstract
categorical-theoretic literature, thereby leading towards a
categorical description of systems. We only take the first step in
this direction, laying out the guidelines for future research. All
of the categories considered here are locally small.

Grandis in \cite{grandis2013homological} introduced the notion of
homological categories to investigate homological properties of
various categories such as the category of modules over a semiring
or the category of pairs of topological spaces. See
\cite[pp.~9]{grandis2013homological} for a comprehensive list of
examples.

We conjecture that the category of systemic ${\mathcal
    A}$-modules yields a homological category in the sense of Grandis
\cite{grandis2013homological}.   The idea is to view everything in
terms of morphisms, so $\mathbf A$ would be replaced by Hom sets,
and $\tT$ would be replaced by distinguished subsets.   In this
treatment the null morphisms play
a prominent role. 

\subsection{Categories with negation}$ $


We start by considering the negation map categorically. We assume, generalizing ``preadditive,'' that $\Hom
(A,B)$, the set of morphisms from $A$ to $B$, is an additive
semigroup for all objects $A,B$.

\begin{defn}
Let $\mathcal{C}$ be a category.
A \textbf{categorical ideal of morphisms} of $\mathcal{C}$ is a set $N$ of morphisms
such that, for $f \in N$, the composite $hfg$ (when it is defined)
is also in $N$ for any morphisms $h,
g$.\footnote{\cite[\S~1.3.1]{grandis2013homological} calls this an
``ideal'' but we prefer to reserve this terminology for semirings.}
The  morphisms in $N$ are called \textbf{null morphisms}.
\end{defn}

\begin{defn} \label{Definition; negation functor}
Let $\mathcal{C}$ be a category with a categorical ideal $N$. A \textbf{negation functor} on $\mathcal{C}$ is an endofunctor   $(-): \mathcal C \to \mathcal
    C $ satisfying $(-)\mathbf A = \mathbf A$ for each object $\mathbf
    A,$ and, for all morphisms $f,g$:
    \begin{enumerate}\label{negpro}\eroman
        \item  $(-)((-)f)= f$.
        \item $(-)(fg) = ((-)f)g = f((-)g)$, i.e., any composite of morphisms
        (if defined) commutes with $(-)$.
        \item  $f(-)f \in N$ for any $f$.\footnote{By this notation,
        we mean $f+(-)f$. This makes sense since we assume that $\Hom(A,B)$ is an additive semigroup.
        Alternatively, more in line with \cite{CC2}, one could take $N$ to be $\{ f: f = (-)f \}.$
        This is more inclusive since $f^\circ = (-) f^\circ.$ On the
        other hand when $(-)$ is the identity, $f = (-)f$ always.}
   \end{enumerate}
    \textbf{Unique negation} means $f(-)g \in N$ implies
    $g=f.$
\end{defn}

\begin{remark}
    Note that being a negation functor depends on a choice of a categorical ideal $N$ because of Definition \ref{Definition; negation functor} (iii).
\end{remark}

When lacking a negation functor, we obtain one via the categorical
version of symmetrization, performing Example~\ref{semidir38}  at
the level of $\Hom,$ and thus also obtaining a symmetrization
functor.

\begin{defn}
Let $\mathcal{C}$ be a pointed category. The symmetrized category $\widehat {\mathcal{C}}$ of $\mathcal{C}$ consists of the following.
\begin{enumerate}
    \item
Objects; the pairs $(A,A)$ for each object $A$ of $\mathcal C$.
\item
Morphisms; for $(A,A), (B,B) \in \emph{Obj}(\widehat{\mathcal C})$,
\[
\Hom_{\widehat{\mathcal{C}}}((A,A),(B,B))=\{(f,g) \mid f,g \in \Hom_\mathcal{C} (A,B)\}.
\]
\end{enumerate}
The composition of two morphisms are given by the \textbf{twist
product}\footnote{See \cite[Equation (22)]{CC2}]. This is a
categorical version of the twist product defined for modules in
Definition \ref{wideh7}. For application to tropical geometry, see \cite{JoM}.}, that is, for $(f_1,f_2):(B,B) \to (C,C)$
and $(g_1,g_2):(A,A) \to (B,B)$,
\begin{equation}\label{eq: twist morphism}
(f_1,
f_2)(g_1,g_2) = ( f_1 g_1 + f_2 g_2,  f_2 g_1+  f_1 g_2)
\end{equation}
\end{defn}

\begin{prop}
Let $\mathcal{C}$ be a pointed category. Then, $\widehat{\mathcal{C}}$ is a category.
\end{prop}
\begin{proof}
For each object $(A,A)$, one can easily check that the morphism $(1_A,0_A):(A,A) \to (A,A)$ is the identity morphism. The composition defined by \eqref{eq: twist morphism} is associative since
 $$\begin{aligned}
 ((f_1,
f_2)(g_1,g_2))(h_1,h_2) & = ( (f_1 g_1 + f_2 g_2)h_1 + ( f_2 g_1+
f_1 g_2)h_2, ( f_1 g_1 + f_2 g_2)h_2 + ( f_2 g_1+  f_1 g_2)h_1) \\
& = ( f_1 (g_1 h_1  +   g_2 h_2) +  f_2( g_1 h_2+   g_2h_1),   f_1
(g_1 h_2 + g_2h_1) +  f_2( g_1 h_1+   g_2  h_2)) \\ & = (f_1,
f_2)((g_1,g_2)(h_1,h_2))
\end{aligned}$$
\end{proof}

\begin{defn}
Let $\mathcal{C}$ be a pointed category and $\widehat{\mathcal{C}}$ be the symmetrized category of $\mathcal{C}$. We define the following endofunctor, called the \textbf{switch}, on $\widehat{\mathcal{C}}$; for any $(f_1,f_2) \in \Hom_{\widehat{\mathcal{C}}}((A,A),(B,B))$, \[
(-)_{\operatorname{sw}}(f_1, f_2) = (f_2,f_1).
\]
\end{defn}

\begin{remark}
As we mentioned above, we assume that the set $\Hom_\mathcal{C}(A,B)$ is an additive semigroup for any objects $A,B$ of $\mathcal{C}$. This will induce additive semigroup structure on $\Hom_{\widehat{\mathcal{C}}}((A,A),(B,B))$ (in a coordinate-wise way).
\end{remark}

\begin{prop}
With the same notation as above, $(-)_{\operatorname{sw}}: \widehat{\mathcal{C}} \to \widehat{\mathcal{C}}$ is a negation functor with the categorical ideal:
\[
N_{\widehat {\mathcal{C}}} = \{(f,f): f \in \Hom_{\mathcal C} (A,B), ~\forall A,B \in \emph{Obj}(\mathcal{C})\}.
\]
\end{prop}
\begin{proof}
We first claim that $N_{\widehat {\mathcal{C}}}$ is a categorical ideal of $\widehat{\mathcal{C}}$. Indeed, we have
\[
(f_1,f_2)(g,g) = (f_1 g + f_2 g, f_1 g + f_2 g) \in
N_{\widehat {\mathcal{C}}},
\]
and
\[
(g,g)(f_1,f_2) = (g f_1 + g f_2 , g f_1 + g f_2) \in N_{\widehat {\mathcal{C}}},
\]
proving $N_{\widehat {\mathcal{C}}}$ is a categorical ideal of
morphisms.

Next, we prove that the functor $(-)_{\operatorname{sw}}$ satisfies the conditions in Definition \ref{Definition; negation functor}. Clearly we have
\[
(-)_{\operatorname{sw}}((-)_{\operatorname{sw}}(f,g))=(-)_{\operatorname{sw}}(g,f)=(f,g).
\]
Now, consider $(f_1,f_2),(g_1,g_2)$ such that $f_i:B \to C$ and $g_i:A \to B$ for $i=1,2$. Then, we have
\[
(-)_{\operatorname{sw}}((f_1,f_2)(g_1,g_2))=(-)_{\operatorname{sw}}(( f_1 g_1 + f_2 g_2,  f_2 g_1+  f_1 g_2))=(f_2 g_1+  f_1 g_2,f_1 g_1 + f_2 g_2 )
\]
\[
=(f_2,f_1)(g_1,g_2)=((-)_{\operatorname{sw}}(f_1,f_2))(g_1,g_2).
\]
Similarly, one can show that
\[
(-)_{\operatorname{sw}}((f_1,f_2)(g_1,g_2))=(f_1,f_2)((-)_{\operatorname{sw}}(g_1,g_2)),
\]
showing the second condition.

Finally, we have
\[
(f,f)(-)_{\operatorname{sw}}(f,f)=(f,f)+(f,f)=(f+f,f+f) \in N_{\widehat {\mathcal{C}}}.
\]
\end{proof}

By combining the above results, we conclude the following.

\begin{thm}\label{catview}
Let $\mathcal{C}$ be a pointed category. Suppose that $\mathcal{C}$ is equipped with a negation functor $(-)$ with respect to a categorical ideal $N$. Then, there is a faithful functor $F:{\mathcal{C}} \to \widehat {\mathcal{C}}$ such that $F(A)=(A,A)$ for an object $A$, and $F(f)=(f,0_A)$ for $f \in \Hom_\mathcal{C} (A,B)$.
\end{thm}
\begin{proof}
 We first prove that $F$ is a functor. In fact, $F(1_A)=(1_A,0_A)=1_{(A,A)}$. Furthermore,
 \[
 F(gf)=(gf,0)=(g,0)(f,0)=F(g)F(f),
 \]
showing that $F$ is a functor. One can easily see that $F$ is faithful.
\end{proof}
%

\subsection{Towards semiexact categories and homological
categories}\label{homcat}$ $

  In this subsection, we briefly explore how our framework of systems is related to more categorical framework of Grandis \cite{grandis2013homological} as well as Connes and Consani's work \cite{CC2}. In particular, we prove that the category of $\mathcal{A}$-modules is semiexact (under certain condition), which we expect to be a homological category. We further investigate the category of systemic modules in this context. To make this section self-contained, we briefly recall the definition of semiexact category and homological category.

\begin{defn}
Let $\mathcal{C}$ be a category.
\begin{enumerate} \eroman
\item
A categorical ideal $N$ is said to be \textbf{closed} if there
exists a set $\mathcal{O}$ of objects in $\mathcal{C}$ such that
\begin{equation}\label{eq:cld ideal}
N=\{f \in \emph{Mor}(\mathcal{C}) \mid
\textrm{ $f$ factors
    through some object in $\mathcal O$}\}.
\end{equation}
\item
An \textbf{N-category} is a category with a specified categorical
closed ideal $N$ of morphisms.
\end{enumerate}
\end{defn}

From the systemic point of view, the  set $\mathcal{O}$ is comprised
of objects whose elements are quasi-zeros as the following example illustrates.

\begin{example} \label{example: A-mod}
Let $(\mathcal A,\tT, (-))$  be a triple, and
$\textbf{Mod}_\mathcal{A}$ be the category of systemic
$\mathcal{A}$-modules. For each $\mathcal{M}\in
\textbf{Mod}_\mathcal{A}$, recall that we have a submodule:
\[
\mathcal{M}^\circ=\{a(-)a \mid a \in \mathcal{M}\}.
\]
Consider the following set
\[
\mathcal{O}=\{\mathcal{M}^\circ \mid \mathcal{M} \in \textbf{Mod}_\mathcal{A}\}.
\]
It is clear that $N$ defined as in \eqref{eq:cld ideal} is a closed categorical ideal of $\textbf{Mod}_\mathcal{A}$. Hence, $\textbf{Mod}_\mathcal{A}$ is an $N$-category.
\end{example}
%
%
%
%
%
%
%

One ingredient of homological category is to replace the zero object (or zero morphism) of an abelian category with the set of zero objects (or zero morphisms) through a closed categorical ideal $N$. To be precise, Gradis introduced the following:

\begin{defn}
Let $(\mathcal C, N)$ be an $N$-category. Let $f:X \to Y$ be a morphism of $\mathcal{C}$.
\begin{enumerate}\eroman
\item
The \textbf{N-kernel}~\footnote{In \cite{grandis2013homological},
this is called the kernel with respect to $N$.} of $f$ is a morphism
$\ker f: \emph{Ker} f \to X$ satisfying the usual universal property with
respect to $N$, that is, $\ker f$ satisfies the following conditions:
\begin{enumerate}
    \item
The composite $f(\ker f)$ is null, i.e., $f(\ker f) \in N$.
    \item
For any morphism $h$, if $fh$ is defined and null, then $h$ uniquely factors through $\ker f$.
\end{enumerate}
\item
The \textbf{N-cokernel} of $f$ is a morphism $\coker f: Y \to \emph{Coker}(f)$ (suitably defined) satisfying the usual universal property with respect to $N$: The composite  $f(\coker f) \in N,$ and for any morphism $h$, if $hf$ is defined and null, then $h$ uniquely factors through $\coker f$.
\item
A morphism $f$ is said to be \textbf{N-monic} (resp.~ \textbf{N-epic})  if $fh$ (resp.~$kf$)  is null,  then $h$ is null (resp.~$k$ is null) for any morphism $h$ (resp.~$k$) where $fh$ (resp.~ $kf$) is defined.
\item
A \textbf{normal} N-monic (resp.~\textbf{normal} N-epic) is an N-monic (resp.~N-epic) which is an N-kernel (resp.~N-cokernel) of some morphism.
\end{enumerate}
\end{defn}

We now recall the definition of semiexact category and
homological category. We then prove that the category of
$\mathcal{A}$-modules (satisfying some mild condition) is semiexact.
We leave open the problem of proving (or disproving) that the
category of $\mathcal{A}$ modules (with some suitable condition) is
homological.

Let $(\mathcal C, N)$ be an $N$-category. Suppose that a morphism $f:X \to Y$ has an N-kernel and N-cokernel. Then $f$ uniquely factors through its \emph{normal coimage} $\textrm{Ncm} f:=\coker(\ker f)$ and \emph{normal image} $\textrm{Nim} f:=\ker(\coker f)$ as follows:
\begin{equation}\label{diagram}
\begin{tikzpicture}[baseline=-0.8ex]
\matrix(m)[matrix of math nodes, row sep=3em, column sep=3.5em, text
height=1.5ex, text depth=0.25ex]
{\textrm{Ker}  f&X&Y & \textrm{Coker} f\\
    \textrm{ }& \textrm{Ncm} f& \textrm{Nim} f\\}; \path[->,font=\scriptsize] (m-1-1) edge
(m-1-2) (m-1-2) edge node[auto] {f} (m-1-3) edge node[auto]
{p}(m-2-2) edge (m-2-3) (m-2-3) edge node[auto] {m}(m-1-3) (m-1-3)
edge (m-1-4) edge (m-2-2) (m-2-2) edge[dotted] node[auto] {g}
(m-2-3);
\end{tikzpicture}
\end{equation}

\begin{defn}\label{exact}
    A morphism $f:X \to Y$ is \textbf{exact} (in the sense of Grandis) if there exists an isomorphism $g$
    which makes \eqref{diagram} commute.\footnote{Once it exists, $g$ as in the above diagram is unique. See \cite[\S 1.5.5]{grandis2013homological}.}
\end{defn}

Now, we  recall the definition of a homological category from \cite{grandis2013homological}.

\begin{defn}\cite[\S 1]{grandis2013homological}\label{semiandexct}
    Let $(\mathcal C, N)$ be an $N$-category.
    \begin{enumerate}
        \item
        $(\mathcal C, N)$ is said to be a \textbf{semiexact} category if every morphism $f:X\to Y$ has an N-kernel and N-cokernel.
        \item
        $(\mathcal C, N)$ is said to be \textbf{homological}  if the following conditions hold:
        \begin{enumerate}
            \item
            $(\mathcal{C},N)$ is semiexact.
            \item
            Normal monics and normal epics are closed under composition.
            \item
            For any normal monic $m:M \to X$ and normal epic $q:X \to Q$
            such that $m \geq \ker q$ (i.e., $m$ factors through $\ker q$), the
            morphism $qm$ is exact in the sense of Definition \ref{exact}.
        \end{enumerate}
    \end{enumerate}
\end{defn}

\begin{example}[{\cite[\S1.4.2]{grandis2013homological}}]\label{grandisex} Let $\mathfrak{Set}_2$ be the category of pairs of sets, i.e., an object is a pair $(A,B)$ of a set $A$ and a subset $B$ of $A$. A morphism from $(A,B)$ to $(A',B')$ is a function $f:A \to A'$ such that $f(B) \subseteq B'$. A morphism $f:(A,B) \to (A',B')$ is null if and only if $f(A)\subseteq
    B'$. Then $\mathfrak{Set}_2$ becomes a homological category.
\end{example}

Now, we prove that the category of $\circ$-idempotent
$\mathcal{A}$-modules is semiexact.

\begin{thm}\label{proposition: modules semiexact}
Let $(\mathcal A,\tT, (-))$ be a triple, and
$\textbf{Mod}_{\mathcal{A},h}$ be the subcategory of
$\textbf{Mod}_\mathcal{A}$ with:
    \begin{enumerate}
        \item
    Objects: The $\circ$-idempotent $\mathcal{A}$-modules, cf., Definition~\ref{signdec1}.
       \item
    Morphisms: homomorphisms of $\mathcal{A}$-modules.
    \end{enumerate}
    With $N$ as in Example \ref{example: A-mod}, $\textbf{Mod}_{\mathcal{A},h}$ is a semiexact category.
\end{thm}
\begin{proof}
Let $f: X \to Y$ be a homomorphism of $\mathcal{A}$-modules. As
explained in Example \ref{precmain0}, we may impose a surpassing
relation $\preceq_\circ$ on $X$ and $Y$ as follows:
\begin{equation}\label{eq: add surpassing relation}
c \preceq_\circ c' \iff c+b^\circ = c'.
\end{equation}
For notational convenience, we will just write
$\preceq_\circ=\preceq$. We will consider all $\mathcal{A}$-modules
in this way.

With \eqref{eq: add surpassing relation}, we first claim that the following is an $N$-kernel of $f$:
\begin{equation}\label{eq: kernel}
\ker f:\ker_{\Mod,X}{f} \longrightarrow X.
\end{equation}
Note that \eqref{eq: kernel} makes sense since $\ker f$ is a homomorphism and $\ker_{\Mod,X}{f}$ is a submodule of $X$. To prove our claim, for any $x \in \ker_{\Mod,X}{f}$, we have $f(x) \succeq \zero$, that is, $f(x)=b(-)b$ for some $b \in Y$. In particular, $f(x) \in Y^\circ$. Therefore, we have the following commutative diagram:
\begin{equation}
\begin{tikzcd}
\ker_{\Mod,X}{f} \arrow{r}{\ker f}\arrow{rd}{f}
&X \arrow{r}{f}
&Y \\
&Y^\circ \arrow[hook]{ru}
\end{tikzcd}
\end{equation}
showing that $f(\ker f)$ is null.

Next, suppose that $h:Z \to X$ is a homomorphism such that $fh$ is null. Since $fh$ is null, there exists an $\mathcal{A}$-module $\mathcal{N}$ such that the following commutes:
\[
\begin{tikzcd}
Z \arrow{r}{h}\arrow{rd}[swap]{\varphi}
&X \arrow{r}{f}
&Y \\
&\mathcal{N}^\circ \arrow{ru}[swap]{\psi}
\end{tikzcd}
\]
for some homomorphisms $\varphi$ and $\psi$. Indeed, one can
observe, for each $z \in Z$, $f(h(z)) \in Y^\circ$ since $fh=\psi
\varphi$ and $\varphi(z) \in \mathcal{N}^\circ$. In
particular, $h(z) \in \ker_{\Mod,X}{f}$, and hence $h$ uniquely
factors through $\ker f$. This proves that $\ker f:\ker_{\Mod,X}{f} \longrightarrow X$ is an $N$-kernel of $f$.

Now, we prove that the $N$-cokernel exists. We use the idea of Connes and Consani in \cite{CC2} to define the $N$-cokernel. For a given homomorphism $f:X \to Y$, we define
\[
\textrm{Coker}f:=Y/\Cong,
\]
where for $y, y' \in Y$,
\begin{equation}\label{eq: cokernel condition1}
(y , y') \in \Cong \iff g(y)=g(y') ~\forall Z ~\forall g \in
\Hom(Y,Z)  \textrm{ such that }f(X) \subseteq
\ker_{\Mod,X}{g}.
\end{equation}
We let $\coker f:Y \to \textrm{Coker}f$ be the projection map. This definition makes sense from the following two observations:
\begin{enumerate}
    \item
$\Cong$ is a congruence relation since $g$ is a homomorphism.
Furthermore, for any $[y] \in Y/\Cong$, where $[y]$ is the
equivalence class of $y \in Y$ in $Y/\Cong$, we have that
$([y]^\circ)^\circ = [(y^\circ)^\circ]=[y^\circ]=[y]^\circ$. Hence,
$\textrm{Coker}f=Y/\Cong$ is $\circ$-idempotent.
\item
The projection $\coker f:Y \to Y/\Cong$ is a homomorphism.
\end{enumerate}

Now, we claim that for any $c \in X$, we have
\[
[f(c)]=[f(c)^\circ].
\]
In fact, suppose that $g \in \Hom(Y,Z)$ such that $f(X) \subseteq
\ker_{\Mod,X}{g}$. This implies that $gf(X) \subseteq Z^\circ$.
Hence, $gf(c)=b^\circ$ for some $b \in Z^\circ$. It follows that
\[
g(f(c)^\circ) = gf(c)^\circ=(b^\circ)^\circ=b^\circ=gf(c).
\]
In particular, we have   $[f(c)]=[f(c)^\circ]$. This shows the
following diagram commutes, proving that $(\coker f) f \in N$:
\begin{equation}
\begin{tikzcd}
X \arrow{r}{f}\arrow[swap]{rd}{f^\circ}
&Y \arrow{r}{\coker f}
&Y/\Cong \\
&Y^\circ \arrow[swap]{ru}{(\coker f)|_{Y^\circ}}
\end{tikzcd}
\end{equation}

Next, suppose that $hf \in N$ for some $h:Y \to Z$. Then, there exists a module $\mathcal{M}$ and the following commutative diagram:
\begin{equation}
\begin{tikzcd}
X \arrow{r}{f}\arrow[swap]{rd}{\varphi}
&Y \arrow{r}{h}
&Z \\
&\mathcal{M}^\circ \arrow[swap]{ru}{\psi}
\end{tikzcd}
\end{equation}

Since $hf$ factors through $\mathcal{M}^\circ$, we have that $f(X) \subseteq \ker_{\Mod,X}{h}$. In particular, from \eqref{eq: cokernel condition1}, the following map is well-defined:
\[
\tilde{\varphi}: \textrm{Coker}f \longrightarrow Z, \quad [y] \mapsto h(y).
\]
Furthermore, in this case, we have that $h =  \tilde{\varphi}(\coker f)$, showing that $h$ factors through $\textrm{Coker}f$. One can easily see that this is unique.

This completes the proof that $\textbf{Mod}_{\mathcal{A},h}$ is semiexact.

\end{proof}

\begin{rem}
In the proof of Theorem \ref{proposition: modules semiexact},
 one may observe that for the $N$-kernel, we do not need the assumption
  that morphisms are homomorphisms or that our modules are $\circ$-idempotent.
  On the other hand, the $N$-cokernel is trickier than the $N$-kernel. For instance, as
we mentioned before, for a $\preceq$-morphism $f:X \to Y$, the image
$f(X)$ does not have to be a submodule of $Y$ since we only have
$f(x+y) \preceq f(x)+f(y)$ which is not enough to ensure that
$f(x)+f(y) \in f(X)$. Towards this end, we have to restrict
ourselves exclusively to homomorphisms. Furthermore,
$\circ$-idempotence of modules is essential in our proof to show
that $(\coker f) f \in N$.
\end{rem}

\begin{remark}
In \cite[Theorem 6.12]{CC2}, Connes and Consani proved that the
category of $\mathbb{B}$-modules endowed with an involution is a
homological category.  Much of their theory can be restated for
modules over triples. We expect that the category of
$\mathcal{A}$-modules in Example \ref{example: A-mod} may be
homological. Although we do not pursue this in this paper, one may
be able to prove this by closely following the proofs in \cite{CC2}.
\end{remark}

Since $\preceq_\circ$ is a special case of surpassing relation, one
may ask whether or not the category of systemic modules is
semiexact. To some extent, this seems to be doable, for example for
convex images.

Let $\mathcal{A}$ be a system, and $\mathcal{C}$ be the category of systemic $\mathcal{A}$-modules. For each $\mathcal{M}\in \mathcal{C}$, recall that we have
\[
\mathcal{M}_{\emph{Null}}=\{a \in \mathcal{M} \mid a \succeq \zero\}.
\]
Consider the following set
    \[
    \mathcal{O}=\{\mathcal{M}_{\emph{Null}} \mid \mathcal{M} \in \mathcal{C}\}.
    \]
    It is clear that $N$ defined as in \eqref{eq:cld ideal} is a closed categorical ideal of $\mathcal{C}$. Hence, $\mathcal{C}$ is an $N$-category.

\begin{prop}\label{proposition: N-kernel}
In $\mathcal{C}$, any $\preceq$-morphism $f:X \to Y$ has an $N$-kernel given as follows:
\[
\ker f:\ker_{\Mod,X}{f} \longrightarrow X.
\]
\end{prop}
\begin{proof}
The same proof as in Theorem \ref{proposition: modules semiexact}
works.

\end{proof}

\begin{prop}\label{proposition: N-cokernel systemic}
Let $\mathcal{C}_{h}$ be the category of systemic modules with
homomorphisms, viewed   as an $N$-category with the same $N$. Let
$f:X \to Y$ be a homomorphism. Then, with
 $\coker f:\emph{Coker}f \to Y/\Cong$ as in Theorem \ref{proposition: modules semiexact}, if $hf \in N$ for some $h:Y \to Z$, then $h$ uniquely factors through $\emph{Coker} f$.
\end{prop}
\begin{proof}
The same proof as in Theorem \ref{proposition: modules semiexact}
works.
\end{proof}

 An affirmative answer to the following question could provide a more general context for Connes and Consani's result:

\begin{oprblm}
Let $(\mathcal A,\tT, (-))$ be a triple, or $(\mathcal A,\tT, (-),
\preceq)$ a system. Find general conditions for the category of
$\mathcal A$-modules to be homological.
\end{oprblm}

\end{document}